\documentclass[11pt]{article}

\usepackage{amsmath,amssymb,amsthm} 
\usepackage[unicode,breaklinks=true,colorlinks=true]{hyperref}
\usepackage[capitalise]{cleveref}
\usepackage[dvipsnames]{xcolor}
\usepackage{mathrsfs}
\usepackage{verbatim}
\usepackage{mathtools}
\usepackage[normalem]{ulem}

\usepackage{tikz}
\usetikzlibrary{decorations.pathreplacing}

\usepackage{soul}
\usepackage{marginnote}

\usepackage[top=1in, bottom=1in, left=1.25in, right=1.25in, marginparwidth=1in, marginparsep=0.1in]{geometry}

\numberwithin{equation}{section}

\newtheorem{theorem}{Theorem}[section]
\newtheorem{proposition}[theorem]{Proposition}
\newtheorem{lemma}[theorem]{Lemma}
\newtheorem{definition}[theorem]{Definition} 
\newtheorem{corollary}[theorem]{Corollary}

\theoremstyle{remark}

\definecolor{darkblue}{rgb}{0,0,0.7}
\definecolor{darkred}{rgb}{0.6,0,0}

\newcommand{\al}{\alpha}

\newcommand{\de}{\delta}
\newcommand{\e}{\varepsilon}
\newcommand{\ga}{{\gamma}}

\newcommand{\la}{\lambda}

\newcommand{\si}{\sigma}
\newcommand{\td}{\tilde}

\newcommand{\De}{\Delta}

\newcommand{\Bp}{{\dot B_{p,\I}^{-1+3/p}}}

\newcommand{\BMOD}{\text{BMO}^{-1}}
\newcommand{\R}{{\mathbb R }}

\newcommand{\Kp}{{\mathcal{K}_p}}
\newcommand{\N}{{\mathbb N}}
\newcommand{\Z}{{\mathbb Z}}

\newcommand{\cN}{{\mathcal N}}

\newcommand{\cK}{{\mathcal K}}

\newcommand{\pd}{{\partial}}
\newcommand{\nb}{{\nabla}}

\newcommand{\I}{\infty}
 
\renewcommand{\div}{\mathop{\mathrm{div}}}

\newcommand{\supp}{\mathop{\mathrm{supp}}}

\newcommand{\donothing}[1]{{}}

\newcommand{\EQ}[1]{\begin{equation}\begin{split} #1 \end{split}\end{equation}}
\newcommand{\EQN}[1]{\begin{equation*}\begin{split} #1 \end{split}\end{equation*}}

\DeclareMathOperator*{\esssup}{ess\,sup}

\newcommand{\loc}{\mathrm{loc}}

\newcommand{\uloc}{\mathrm{uloc}}

\begin{document}
\title{
Asymptotic properties of discretely self-similar Navier-Stokes solutions with rough data}
\author{Zachary Bradshaw\footnote{
Zachary Bradshaw, Department of Mathematical Sciences, University of Arkansas, Fayetteville, AR,  USA;
e-mail: \url{zb002@uark.edu}
} \and Patrick Phelps\footnote{Patrick Phelps, Department of Mathematics, Temple University, Philadelphia, PA, USA;
 e-mail: \url{patrick.phelps0001@temple.edu}}
 } 
\date{\today}
\maketitle 
 
\begin{abstract}
In this paper we explore the extent to which discretely self-similar (DSS) solutions to the 3D Navier-Stokes equations with rough data \textit{almost} have the same asymptotics as DSS flows with smoother data.
 In a previous work \cite{BP1}, we established algebraic spatial decay rates for data in $L^q_\loc(\R^3\setminus\{0\})$ for $q\in (3,\I]$. The optimal rate occurs when $q=\I$ and rates degrade as $q$ decreases. In this paper, we show that these solutions can be further decomposed into a term satisfying the optimal $q=\I$ decay rate---i.e.~have asymptotics like $(|x|+\sqrt t)^{-1}$---and a term with the $q<\I$ decay rate multiplied by a prefactor which can be taken to be arbitrarily small. This smallness property is new and implies the $q<\I$ asymptotics should be understood in a little-o sense. 
 The decay rates in \cite{BP1}   break down when $q=3$, in which case spatial asymptotics have not been explored. The second result of this paper shows that DSS solutions with data in $L^3_\loc(\R^3\setminus\{0\})$ can   be expanded into a term satisfying the $(|x|+\sqrt t)^{-1}$ decay rate and a  term that can be taken to be arbitrarily small in a scaling invariant class.  A Besov space version of this result is also included.

\end{abstract}

\tableofcontents

\section{Introduction}
In this paper, we continue work in \cite{BP1} developing  asymptotic properties of discretely self-similar Navier-Stokes flows. The  Navier-Stokes equations model the velocity $u:\R^3\times(0,\I) \to \R^3$ and pressure $p:\R^3\times(0,\I) \to \R$ of a viscous incompressible fluid and can be written as
\EQ{\label{eq:NS}
\pd_t u -\De u +u\cdot \nb u + \nb p =0, \quad \div u =0,
}
where we have taken forcing to be zero and viscosity to be unitary. This is paired with the initial condition
\EQ{
	u(\cdot, 0) = u_0, \quad \div u_0 = 0.
}
All equations are understood distributionally. 
 The Navier-Stokes equations have a parabolic scaling property:
If $u$ is a solution with associated pressure $p$, then the pair $u^\la$, $p^\la$ defined by
\EQ{u^\la(x,t):= \la u(\la x, \la^2 t), \,\, p^\la(x,t):= \la^2 p(\la x, \la^2 t),}
also solves \eqref{eq:NS}. We investigate discretely self-similar (DSS) solutions which have the property $(u,p)=(u^\la,p^\la)$ for at least some $\la>1$. Data is called DSS if the same property holds with the time variable omitted. The class of self-similar solutions is a stronger class for which this property holds for all $\la>0$. These classes are interesting cases to investigate non-uniqueness in the Leray-Hopf class \cite{Jia-uniqueness,GuillodSverak,ABC}, and to explore the possible  failure of eventual regularity of Lemari\'e-Rieusset-style solutions with data in ultracritical classes \cite{BT1}. The existence of self-similar and DSS solutions is known in a variety of cases \cite{AB,Barraza,BT1,BT3,BT5,CP,Chae-Wolf,GiMi,JS,Kato,KT-SSHS,LR2,Tsai-DSSI,FDLR}, of which we are most concerned with large DSS data in $L^{3,\I}$ \cite{BT1}, the critical Besov spaces \cite{BT3,AB} and  $L^2_\loc$ \cite{Chae-Wolf,BT5,FDLR}. The existence of DSS solutions in the Koch-Tataru space $\BMOD$ \cite{KT} is not known---this is an interesting and apparently difficult open problem.

Regularity for self-similar solutions is well-known. In \cite{Grujic}, Gruji\'c gives an elegant argument that any forward self-similar suitable weak solution is smooth. Gruji\'c's argument does not hold in general for DSS solutions. In some cases smoothness is known, e.g.,~for $u_0\in L^{3,\I}$ and $\la$ close to $1$ \cite{KMT} or for small data $u_0\in L^2_\loc$ \cite{BT8}. In \cite{KMT2}, the $\la$-close-to-$1$ argument is extended to a space marginally smaller than $L^2_\uloc$. When $\la$ is not close to $1$, using local smoothing DSS solutions in the local energy class with data in $E^2$ (see Section 2.4) are shown to be regular on the set
\[\{(x,t)\in\R^3\times (0,\I): |x|\ge R_0\sqrt t\},\]
where $R_0$, the radius of far-field regularity, depends on the data. Note that $L^{3,\I}\subset E^2$.

In \cite{BP1}, algebraic decay rates are established for DSS local energy solutions with data $u_0\in L^q_\loc(\R^3\setminus \{0\})$, $q>3$.  The optimal decay of $(|x|+\sqrt t)^{-1}$  is achieved for $q=\I$. This is optimal in the sense that, even if the data is more regular, e.g., locally H\"older continuous away from $x=0$, the leading order spatial asymptotics remain $O((|x|+\sqrt t)^{-1})$.  The main result about decay rates for $L^q_\loc(\R^3\setminus \{0\})$ in \cite{BP1} is as follows.

\begin{theorem}[Bradshaw \& Phelps \cite{BP1}]  \label{theorem:Lpdecay1} Let $q\in (3,\I]$ and assume ${u_0\in L^q_\loc (\R^3\setminus \{0\})}$ is divergence free and $\la$-DSS. Assume $u$ is a DSS local energy solution with initial data $u_0$.
 It follows that:
\begin{enumerate}
\item For  $|x|\geq R_0\sqrt t$,
\EQ{\label{ineq:pointwiseBound.q}|u(x,t)|\lesssim_{u_0,q,\la} \frac{1}{\sqrt{t}^{\frac{3}{q}} \left( |x|+\sqrt{t}\right) ^{1-\frac{3}{q}}}.} 
\item For the same selection of $x$ and $t$, the difference $\td u := u-e^{t\De}u_0$ satisfies,  
\EQ{\label{ineq:improved.decay.for.difference}
|\td u(x,t)| \lesssim_{u_0,q,\la} \frac{1}{\sqrt t^{\frac{6}{q}-1} (|x|+\sqrt{t})^{2-\frac{6}{q}}}.}  
\end{enumerate}
\end{theorem}

The estimate for $\td u$ was then pushed to higher Picard iterates. {Stronger bounds were also obtained for smoother initial data but these are not relevant to the present work.} Picard iterates are defined as follows: $P_0=e^{t\Delta}u_0$; $P_{k+1}= P_0+B(P_k,P_k)$, where 
\[
B(f,g)= -\frac 1 2 \int_0^t e^{(t-s)\Delta}\mathbb P \nb \cdot (f\otimes g + g\otimes f)\,ds.
\] 
This bilinear operator is further discussed  in Section 2. 
The second result in \cite{BP1} is the following theorem. 
\begin{theorem}[Bradshaw \& Phelps \cite{BP1}] \label{theorem:pwdecay}
Let $q\in (3,\I]$ and assume ${u_0\in L^q_\loc (\R^3\setminus \{0\})}$ is divergence free and $\la$-DSS for some $\la>1$. Assume $u$ is a $\la$-DSS local energy solution with initial data $u_0$. 
Define for $k\in \N_0$,
\[
a_k = (k+2)\bigg(1-\frac{3}{q}\bigg) = a_{k-1}+1-\frac 3 q;   \quad k_q=\left\lceil \frac {4q} {q-3}-2  \right\rceil.
\]
   The following hold:
\begin{enumerate}
\item 
We have  for $|x|\geq R_0\sqrt t$ and $k< k_q$ that 
	\EQ{|u-P_k|(x,t) \lesssim_{k,\la,R_0,u_0}\frac{\sqrt{t}^{a_k}}{\sqrt{t}(|x|+\sqrt{t})^{a_k}}.}
\item We have for $|x|\geq R_0\sqrt t$ and $k\geq k_q$ that 
\EQ{ | u-P_{k}|(x,t) \lesssim_{k,\la,R_0,u_0} \frac{\sqrt{t}^3}{(|x|+\sqrt{t})^4}.} 
\end{enumerate}
\end{theorem}

As an application, the separation of hypothetical non-unique DSS solutions was bounded in \cite{BP1}. In particular, if $v$  is another DSS local energy solution with data $u_0$, then  for  $|x|\geq R_0\sqrt t$, 
\EQ{\label{lim:seperation.scale}
 | u-v |(x,t) \lesssim_{{q},\la,R_0,u_0} \frac{\sqrt{t}^3}{(|x|+\sqrt{t})^4}, 
} 
which is a sort of local stability away from $x=0$. A similar result was proven for general flows in \cite{BP2}. 

Our goals in this paper are to   to refine the asymptotics developed in \cite{BP1} when $q\in (3,\I)$ and explore what asymptotics are available when $q=3$.
Our first new result states that, for any $q>3$, the optimal $(|x|+\sqrt t)^{-1}$ decay rate holds up to an arbitrarily small multiple of $\sqrt t^{-\frac 3 q}(|x|+\sqrt t )^{\frac 3 q -1}$. 
   
\begin{theorem}[Little-o algebraic decay]  \label{thrm:Lpdecay1} Let $q\in (3,\I)$ and assume ${u_0\in L^q_\loc (\R^3\setminus \{0\})}$ is divergence free and $\la$-DSS for some $\la>1$. Assume $u$ is a $\la$-DSS local energy solution with initial data $u_0$.  {Let $R_0>0$ be the smallest number so that  $u$ is smooth on $|x|\geq R_0\sqrt t$.}  For any $\e>0$ there exists $C_{*}(\e,u_0)>0$ and  vector fields $a$ and $b$ so that  \[
u(x,t) = a(x,t)+b(x,t),
\]
and, for $|x|\ge R_0 \sqrt{t}$, 
\EQ{
|a(x,t)| \leq \frac {2\e} {\sqrt t^{\frac 3 q} (|x|+\sqrt t)^{1-\frac 3 q} };\quad |b(x,t)|  \leq \frac   {C_{*} (\e,u_0)  } {|x|+\sqrt t}.
}
Consequently,
\EQ{\label{lim:vanishing}
\lim_{r\to \I} \sup_{|x|\geq r\sqrt t} |u(x,t)|\sqrt t^{\frac 3 q} (|x|+\sqrt t)^{1-\frac 3 q} = 0.
}
Furthermore, the leading order asymptotic properties are captured by the Picard iterates $P_k$ in the sense that, letting $a_k =\min \{(k+1)(1-\frac 3q),4\}$, if  $k\in \N_0$  and $a_k<4$, then
\EQ{
\label{lim:vanishing3}\lim_{r\to \I} \sup_{|x|\geq r\sqrt t}  |u-P_k|(x,t) \sqrt t^{1-a_{k}} (|x|+\sqrt{t})^{a_{k}} = 0. 
}

\end{theorem}

\noindent{\bf Comments on \cref{thrm:Lpdecay1}:} 
\begin{enumerate}
\item Because self-similarity implies discrete self-similarity,  these observations apply to self-similar solutions as well. In that case, $R_0$ can be taken to be $0$ because self-similar local energy solutions are known to be bounded at positive times \cite{Grujic}. 


\item  
When $a_k\geq 4$ we do not get an improvement over \cite{BP1}. The reason for the exponent $4$ can be found in Lemma \ref{lemma.intbound1}.

\item These can be viewed as both statements about asymptotics as $|x|\to \I$ or, in the case of \eqref{lim:vanishing3}, asymptotics as $t\to 0$ for $x\neq 0$.  Indeed, \eqref{lim:vanishing3} justifies an asymptotic expansion along the lines of \cite[(3.14)]{BP2}. That is, letting $k_*$ denote the largest $k$ so that $a_k<4$, we have 
\EQ{\label{asymptoticExpansion}
u(x,t) = P_0 + \sum_{k=0}^{k_*-1} o\bigg(  \frac{\sqrt{t}^{a_k}}{\sqrt{t}(|x|+\sqrt{t})^{a_k}}  \bigg)   + O \bigg( \frac{\sqrt{t}^3}{(|x|+\sqrt{t})^4}  \bigg),
}
where the little-o and big-O asymptotics apply both as $|x|\to \I$ and as $t\to 0$ (in which case we require $x\neq 0$).
Note that Brandolese provides an asymptotic expansion for self-similar solutions in \cite{Brandolese}, but it is only for \textit{small} data and it depends on $u$.  The expansion \eqref{asymptoticExpansion}  and the corresponding expansion in \cite{BP1} are for large data and, up to the last term, are independent of $u$.

When the Navier-Stokes equations can be viewed as a perturbation of the heat equation, e.g.~in the regimes of Kato \cite{Kato} or Fabes, Jones and Riviere \cite{FJR}, the Picard iterates fully capture the asymptotics of $u$ at $t=0$. The solutions considered in this paper live outside of this regime. The statement \eqref{lim:vanishing3} and the expansion \eqref{asymptoticExpansion} assert that \textit{some} leading order asymptotics are nonetheless captured by Picard iterates.

\item If $v_0\in L^q(\R^3)$ for some $q\in (3,\I)$ and $v$ is Kato's mild solution \cite{Kato}, then 
\[
\sqrt t^{\frac 3 q}\| v(\cdot,t)\|_{L^q}=0. 
\]
Note that \eqref{lim:vanishing} implies that for any $\rho>0$,
\[   \lim_{t\to 0}  \sqrt t^{\frac 3 q} \| u(\cdot,t) \|_{L^q(\{|x|\geq \rho  \}} = 0.\]
This further illustrates the locally subcritical nature of these solutions away from $x=0$. Note that this is less a product of self-similarity than it is a result of of local smoothing.

\end{enumerate}

\bigskip
The algebraic decay rates appearing in the preceding theorems are for $u\in L^{q}_\loc(\R^3\setminus \{0\})\cap DSS$ where $q\in (3,\I]$.
If $u_0$ is DSS and only locally $L^3$ away from $0$, then, in analogy with the heat equation, we do not get an explicit pointwise algebraic decay rate for an ensuing DSS local energy solution (an example illustrating this for the heat equation appears in \cite{BT1}).  This is why Theorem \ref{thrm:Lpdecay1} does not include $q=3$. However, such solutions \textit{almost} enjoy the $O(|x|^{-1})$ pointwise decay rate as the next theorem states. 

Note that $u_0\in L^{3,\I}\cap DSS$ if and only if $u_0\in L^3_\loc(\R^3\setminus \{0\})\cap DSS$ \cite{BT1}. To be more precise, in \cite{BT1} Tsai and the first author showed that, if $u_0$ is $\la$-DSS then
\[
\int_{1\leq |x|\leq \la} |u_0|^3 \,dx\leq 3(\la-1)^2 \| u_0\|_{L^{3,\I}}^3,
\]
and 
\[
\|u_0\|_{L^{3,\I}}^3 \leq \frac {\la^3}{3(\la-1)} \int_{1\leq |x|\leq \la}|u_0|^3 \,dx,
\]
see \cite[(3.5) and (3.7)]{BT1}. By DSS scaling, the finiteness of $\| u_0\|_{L^3(\{1\leq |x|\leq \la\})}$ is equivalent to being in $L^3_\loc(\R^3\setminus \{0\})$. For these reasons, the following theorem can be viewed as providing a statement about asymptotics in the  endpoint case which is excluded from Theorem \ref{thrm:Lpdecay1}. Note that $u_0\in L^{3,\I}\cap DSS$ also implies that $u_0$ is in the homogeneous Herz space $\dot K^0_{3,\I}$ where 
\[
\| u_0\|_{\dot K^s_{p,q}} := \bigg\|  \la^{sk} \| u_0\|_{L^p(A_k)}    \bigg\|_{l^{q}(\Z)},
\]
where $A_k = \{x:\la^k \leq |x|<\la^{k+1}  \}$.
In general, $\dot K^0_{3,\I}\neq L^{3,\I}$, so these can be seen as distinct properties.

\begin{theorem}[Almost algebraic decay in $L^{3}_\loc(\R^3\setminus\{0\})\cap DSS$]\label{thrm:L3decay}
Suppose $u_0$ is divergence free, $\la$-DSS and belongs to $L^{3,\I}$. Let $u$ be a DSS local energy solution with data $u_0$. Then, for any $\e>0$, there exist values $R_0>0$ and $\si\in (0,1/2)$  and  DSS vector fields $a_1 $, $a_2$  and $b$ so that  \[
u(x,t) = a_1(x,t)+a_2(x,t)+b(x,t),
\]
for $|x|\ge R_0 \sqrt{t}$ and
\EQ{ \label{Exp1}
\sup_{0<t<\I}\|a_1\|_{ {\dot K^{0}_{3,\I}}} + \|a_1\|_{\mathcal K_{\I}} <  \e;\quad |a_2(x,t)| \leq \frac   {C_{**} (\e,u_0) } {|x|+\sqrt t};\quad  |b(x,t)|  \leq \frac   {C_{**} (\e,u_0,\si) \sqrt t^{2\si} } {(|x|+\sqrt t)^{1+2\si}},
}
for a constant $C_{**}$  and  where $R_0$ is as in  Theorem \ref{thrm:Lpdecay1} and $\mathcal K_p$ denotes a Kato class which we define in Section \ref{sec:mild.sol}. 
Additionally, for a different choice of $a_1$, $a_2$ and $b$, but still
for $|x|\ge R_0 \sqrt{t}$, we have 
\EQ{\label{Exp2}
\sup_{0<t<\I}\|a_1\|_{{L^{3,\I}}} +& \|a_1\|_{\mathcal K_{\I}} <  \e;\quad |a_2(x,t)| \leq \frac   {C_{**} (\e,u_0,\si) } {|x|+\sqrt t};\quad |b(x,t)|  \leq \frac   {C_{**} (\e,u_0) \sqrt t^{2\si} } {(|x|+\sqrt t)^{1+2\si}}.
} 
Throughout the above $a_1$ and $a_2$ always depend solely on $u_0$.
\end{theorem} 
 
This can be viewed as an endpoint case for our pointwise decay theorem---namely when the initial datum is in $L^3_\loc(\R^3\setminus \{0\})$ in contrast to $L^q_\loc(\R^3\setminus \{0\})$.  Essentially, it says that, ignoring an error that can be made arbitrarily small in a critical space, any DSS solution has the same spatial decay as a DSS local energy solution to \eqref{eq:NS} with data bounded by a multiple of $|x|^{-1}$.  The last term takes this further and says that, away from $x=0$, the initial data uniquely determines the time-asymptotics up to order $t^{\si}$ (excluding $t^{\si}$).

In Section \ref{sec:Besov} we will state and prove a generalization of Theorem \ref{thrm:L3decay} in the critical Besov spaces $\Bp$ where $p \in(3,\I)$.

Let us note that in \cite{BP1} we did not get   decay rates in the endpoint case $q=3$, which corresponds to $u_0\in L^{3,\I}$. The idea in \cite{BP1} is to subtract a uniquely determined vector field from a solution to \eqref{eq:NS} which captures leading order asymptotics. To some extent, Theorem \ref{thrm:Lpdecay1} shows that $a_1$ and $a_2$ capture these asymptotics. An important part of \cite{BP1} is to find subsequent fields which capture higher order asymptotics. It is not clear how to find such subsequent order terms in the present context because $b$ satisfies a perturbed equation with the \textit{critical} drift term $a_1+a_2$. The presence of this critical drift term prevents us from applying the Picard improvement argument used in \cite{BP1} to obtain an expansion like \eqref{asymptoticExpansion}.  
 
Because the vector fields $a_1$ and $a_2$ in Theorem \ref{thrm:L3decay} are determined by $u_0$ and, in particular, do not depend on $u$, Theorem \ref{thrm:L3decay}  provides a new upper bound on the difference of hypothetical non-unique flows. This can be viewed as a confinement of non-uniqueness which extends   \eqref{lim:seperation.scale}.  
\begin{corollary}
    Suppose $u_0$ is divergence free, $\la$-DSS and belongs to $L^{3,\I}$. Let $u$ and $v$ be DSS local energy solutions with data $u_0$. Let $R_0$ and $\si$ be as in Theorem \ref{thrm:L3decay}.
    Then, for $|x|\ge R_0 \sqrt{t}$ we have
    \[
    |u-v|(x,t) \lesssim_{u_0}  \frac  {\sqrt t^{2\si}} {(|x|+\sqrt t)^{1+2\si}}.
\]
\end{corollary}

\begin{proof}Write $u-v$ as $u-a_1-a_2 + a_1+a_2-v$ and use the triangle inequality.
\end{proof}

\bigskip \noindent \textbf{Organization.} In Section 2 we introduce preliminary items including definitions of terms that have appeared in the theorems. In Section 3 we establish splitting results for initial data and solutions to linear problems. In Section 4 we prove the main theorems. In Section 4 we also state and prove an extension of Theorem \ref{thrm:L3decay} in the Besov spaces $\Bp$ for $p\in (3,\I)$.

\section{Solution classes}\label{sec:solutionclasses}

\subsection{Classical function spaces}

Here we introduce function spaces which play an important role in this paper. Solutions in these spaces will then be discussed.

The $L^p$ spaces and $L^p_\loc$ classes are defined in the classical way. We also utilize the shorthand \[L^p_T X :=L^p(0,T;X).\]Uniformly local versions are denoted $L^p_\uloc$ and defined by finiteness of the norm
\[
\| f\|_{L^p_\uloc} :=\sup_{x_0\in \R^3} \| f\|_{L^p(B_1(x_0))}.
\]We denote by $E^p$ the closure of $C_c^\I$ in $L^p_\uloc$. This class is characterized by the condition
\[
\lim_{R\to \I} \| f\|_{L^p_\uloc(\R^3\setminus B_R)}=0.
\]
For $1\leq p<\I$, the endpoint Lorentz spaces $L^{p,\I}$ are defined by finiteness of the quasinorm
\[
\| f\|_{L^{p,\I}} := \sup_{\si>0} \si^p | \{x:\si<|f(x)|  \}|. 
\]
Note that these spaces correspond to the weak-$L^p$ spaces.

\subsection{Littlewood-Paley} We refer the reader to \cite{BCD} for an in-depth treatment of Littlewood-Paley and Besov spaces. Let $\lambda_j=2^j$ be an inverse length and let $B_r$ denote the ball of radius $r$ centered at the origin.  Fix a non-negative, radial cut-off function $\chi\in C_0^\infty(B_{1})$ so that $\chi(\xi)=1$ for all $\xi\in B_{1/2}$. Let $\phi(\xi)=\chi(\lambda_1^{-1}\xi)-\chi(\xi)$ and $\phi_j(\xi)=\phi(\lambda_j^{-1})(\xi)$.  Suppose that $u$ is a vector field of tempered distributions and let $\Delta_j u=\mathcal F^{-1}\phi_j*u$ for $j\geq 0$ and $\Delta_{-1}=\mathcal F^{-1}\chi*u$. Then, $u$ can be written as\[u=\sum_{j\geq -1}\Delta_j u.\]
If $\mathcal F^{-1}\phi_j*u\to 0$ as $j\to -\infty$ in the space of tempered distributions, then we define $\dot \Delta_j u = \mathcal F^{-1}\phi_j*u$ and have
\[u=\sum_{j\in \Z}\dot \Delta_j u.\]
We additionally define
\[
\Delta_{<J} f = \sum_{j<J} \dot \Delta _jf;\quad \Delta_{\geq j} f =f- \Delta_{<J} f,
\]
with the obvious modifications for $\Delta_{\leq J}$ and $\Delta_{>J}$. If we do not specify that $J$ is in integer, then we use $\chi(\la_1^{-1} 2^{J} \xi)$ in the definition of $\Delta_{\leq J}$.

Littlewood-Paley blocks interact nicely with derivatives and, by Young's inequality, $L^p$ norms. This is   illustrated by the Bernstein inequalities which read:
\EQ{\label{ineq.bern}
\| D^\al \dot \Delta _jf \|_{L^p} \lesssim_{\al,p}   2^{j|\al| } \|\dot \Delta _jf\|_{L^p}; \quad   \|   \dot \Delta _jf \|_{L^p} \lesssim_{p,q}   2^{j (\frac 3 q - \frac 3 p) } \|\dot \Delta _jf \|_{L^q},
}
provided $1\leq p\leq q\leq \I$, $\al\in \mathbb N^3$.

The Littlewood-Paley formalism is commonly used to define Besov spaces. 
We are primarily interested in Besov spaces with infinite summability index, the norms of which are
\begin{align*}
&||u||_{B^s_{p,\infty}}:= \sup_{-1\leq j<\infty } \lambda_j^s ||\Delta_j u ||_{L^p(\R^n)},
\end{align*}
and
\begin{align*}
&||u||_{\dot B^s_{p,\infty}}:= \sup_{-\infty< j<\infty } \lambda_j^s ||\dot \Delta_j u ||_{L^p(\R^n)}.
\end{align*} 
The critical scale of endpoint Besov spaces for \eqref{eq:NS} are $\Bp$. Note that $L^3\subset L^{3,\I} \subset \Bp$ for $p\in (3,\I]$. In particular, $\Bp$ contains functions $f$ satisfying $|f(x)|\lesssim |x|^{-1}$ when $p>3$. 

One can change the base $2$ to $\la$ in the preceding definitions without changing the spaces involved---in other words dyadic blocks and $\la$-adic blocks can be used in the definitions to obtain equivalent norms for the same spaces. Supporting details are worked out in \cite{BT3}.

\subsection{Mild solutions}\label{sec:mild.sol}  A mild solution is a solution to \eqref{eq:NS} with the form
\[
u(x,t)=e^{t\Delta}u_0 -\int_0^te^{(t-s)\Delta} \mathbb P \nb \cdot (u\otimes u)\,ds,
\]
which is obtained from Duhamel's formula applied to the following version of \eqref{eq:NS}
\[
\partial_t u -\Delta u = - \mathbb P ( u\cdot \nb u ).
\]
In the above, $\mathbb P$ is the Leray projection operator $\mathbb P f= f- \nabla  \Delta^{-1} (\nabla \cdot f)$. 
The Oseen tensor, $e^{(t-s)\Delta} \mathbb P$, appears in the mild solution formulation. Denoting its kernel by $K$ we have from  Solonnikov \cite{VAS} that
\[
| D^\al K |(x,t) \lesssim_\al \frac 1 { (|x|+\sqrt t)^{3+|\al|}},
\]
where $\al$ is a multi-index. 
We will need some elementary convolution estimates to bound terms which arise from the preceding estimate. The next lemma is exactly \cite[Lemma 2.8]{BP1}.
\begin{lemma} \label{lemma.intbound1}For $a\in [0,5)$ and $b\in[0,2)$ where $a+b<5$ we have 
\EQ{
\int_0^t\int \frac {1} {(|x-y| +\sqrt{t-s})^4} \frac 1 {(|y|+\sqrt s)^a} \frac 1 {\sqrt s^b} \,dy\,ds \leq   \frac C {\sqrt t^{1-a}(|x|+\sqrt t)^a}+\frac C {\sqrt t^{1-4}(|x|+\sqrt t)^4}.
}
\end{lemma}
 The statement and proof of Lemma \ref{lemma.intbound1} closely resemble a similar integral estimate in \cite[Lemma 2.1]{Tsai-DSSI}. We will need this as well. It  reads: Let $0<a<5,\, 0<b<5$ and $a+b>3$. Then 
\EQ{\phi(x,a,b) = \int_0^1 \int_{\R^3} (|x-y|+\sqrt{1-t})^{-a}(|y|+\sqrt{t})^{-b} dydt} is well defined for $x\in \R^3$ and
\EQ{\label{ineq:Tsai.integral}\phi(x,a,b) \lesssim R^{-a} + R^{-b} + R^{3-a-b} [ 1+ (1_{a=3}+1_{b=3})\log R]}
where $R=|x|+2$. These estimates can be extended to other times by the same change of variable in our proof.

Mild solutions generally are not guaranteed to be regular. Indeed, the local energy solutions defined in the next section are also mild solutions \cite{BT7}. 
In the classical literature they were introduced in the context of strong solutions. 
An important line of research concerned what function spaces guarantee global well-posedness of mild solutions for small data. Example of spaces where a positive answer is available are 
\[
L^3 \subset L^{3,\I} \subset \Bp (3<p<\I)\subset \BMOD.
\]
The last space is the Koch-Tataru space  defined  by finiteness of the following norm:
\EQ{\|e^{t\Delta}u_0\|_{\BMOD} := \esssup_{t\in(0,\I)} t^{\frac{1}{2}}\|e^{t\Delta}u_0\|_{L^\I(\R^3)}+\sup_{x\in\R^3}\sup_{R\in(0,\I)}R^{-\frac{3}{2}}\|e^{t\Delta}u_0\|_{L^2(Q_R(x,0))},}
 where the parabolic cylinder $Q_R(x,t)=B_R(x)\times (t-R^2,t)$. The solutions that are guaranteed to exist by the above global well-posedness satisfy  many  useful properties.  For example, if $u_0$ is small in $\Bp$ for some $p\in (3,\I)$, then the unique strong solution evolving from $u_0$ satisfies
\[
\| u \|_{L^{p'} }(t)\leq t^{\frac 3 {2p'} - \frac 1 2} \| u_0\|_{\Bp},
\]
for all $p'\in [p,\I]$ and $t>0$. We give this property some notation: Let $\mathcal K_p$ be the Kato class defined by the finiteness of the norm
\[
\| u \|_{\mathcal K_p} :=\esssup_{t>0} \sqrt{t}^{1-\frac 3 p} \| u(t)\|_{L^p}.
\] 
Since $L^{3,\I}$ embeds continuously in $\Bp$ for $p\in(3,\I)$ we have, for small data $\|u_0\|_{L^{3,\I}}<\e$, that the global strong solution to \eqref{eq:NS} satisfies
\[
\| u\|_{\mathcal K_p} \lesssim \|u_0\|_{L^{3,\I}}\lesssim\e,
\]
for all $p\in (3,\I]$.

Note that self-similar and discretely self-similar initial data can belong to $L^{3,\I}$ but cannot belong to the smaller space $L^3$. Some relations between these functions spaces intersected with the class of DSS vector fields are contained in \cite{BT3}. For example, $L^{3,\I}\subset L^2_\uloc$ but this embedding fails for $\Bp$ when $p>3$. This last fact complicates the proof of a Besov space version of Theorem \ref{thrm:L3decay}. We elaborate on this in Section \ref{sec:Besov}.

\subsection{Local energy solutions}\label{sec:LE.solutions}

In this subsection, we define local energy solutions and compile some known properties that will be needed in what follows. These solutions were introduced by Lemari\'e-Rieusset \cite{LR} and played an important role in the proof of local smoothing in \cite{JS}. Important details of these solutions were worked out by Kikuchi and Seregin in \cite{KS}. Our definition is taken from \cite{BT8}. Because $L^{3,\I}\subset L^2_\uloc$, it is  a natural class in which to  consider some self-similar and DSS solutions \cite{JS,BT1}.

\begin{definition}[Local energy solutions]\label{def:localEnergy} A vector field $u\in L^2_{\loc}(\R^3\times [0,T))$, $0<T\leq \I$, is a local energy solution to \eqref{eq:NS} with divergence free initial data $u_0\in L^2_{\uloc}(\R^3)$, denoted as $u \in \cN(u_0)$, if:
\begin{enumerate}
\item for some $p\in L^{\frac{3}{2}}_{\loc}(\R^3\times [0,T))$, the pair $(u,p)$ is a distributional solution to \eqref{eq:NS},
\item for any $R>0$, $u$ satisfies
\begin{equation}\notag
\esssup_{0\leq t<R^2\wedge T}\,\sup_{x_0\in \R^3}\, \int_{B_R(x_0 )}\frac 1 2 |u(x,t)|^2\,dx + \sup_{x_0\in \R^3}\int_0^{R^2\wedge T}\int_{B_R(x_0)} |\nb u(x,t)|^2\,dx \,dt<\I,\end{equation}
\item for any $R>0$, $x_0\in \R^3$, and $0<T'< T $, there exists a function of time $c_{x_0,R}\in L^{\frac{3}{2}}_{T'}$\footnote{The constant $c_{x_0,R}(t)$ can depend on $T'$ in principle. This does not matter in practice and we omit this dependence.} so that, for every $0<t<T'$  and $x \in B_{2R}(x_0)$  
\EQ{ \label{eq:pressure.dec}
p(x,t)&= c_{x_0,R}(t)-\De^{-1}\div \div [(u\otimes u )\chi_{4R} (x-x_0)]
\\&\quad - \int_{\R^3} (K(x-y) - K(x_0 -y)) (u\otimes u)(y,t)(1-\chi_{4R}(y-x_0))\,dy 
,
}
in $L^{\frac{3}{2}}(B_{2R}(x_0)\times (0,T'))$
where  $K(x)$ is the kernel of $\De^{-1}\div \div$,
 $K_{ij}(x) = \pd_i \pd_j \frac {-1}{4\pi|x|}$, and $\chi_{4R} (x)$ is the characteristic function for $B_{4R}$, 
\item for all compact subsets $K$ of $\R^3$,  $u(t)\to u_0$ in $L^2(K)$ as $t\to 0^+$,
\item $u$ is suitable in the sense of Caffarelli-Kohn-Nirenberg, i.e., for all cylinders $Q$ compactly supported in  $\R^3\times (0,\I)$ and all non-negative $\phi\in C_c^\I (Q)$, we have  the \emph{local energy inequality}
\EQ{\label{ineq:CKN-LEI}
2\iint |\nb u|^2\phi\,dx\,dt \leq \iint |u|^2(\pd_t \phi + \De\phi )\,dx\,dt +\iint (|u|^2+2p)(u\cdot \nb\phi)\,dx\,dt,
}
\item the function
\EQ{
t\mapsto \int_{\R^3} u(x,t)\cdot {w(x)}\,dx
}
is continuous in $t\in [0,T)$, for any compactly supported $w\in L^2(\R^3)$.
\end{enumerate}
\end{definition}

Local energy solutions with data in $E^2$, which is the closure of $C_c^\I$ under the $L^2_\uloc$ norm,  exhibit far-field regularity in the sense that, for every $t>0$ there exists $\rho$ so that $u$ is smooth in the spatial variable for $|x|\geq \rho$.

Local energy solutions are known to satisfy certain \textit{a priori} bounds \cite{LR}. For example, in \cite{JS,BT8}, the following \textit{a priori} bound is proven: 
Let $u_0\in L^2_\uloc$, $\div u_0=0$, and assume $u\in \mathcal N (u_0)$.  For all $r>0$ we have
\begin{equation}\label{ineq.apriorilocal}
\esssup_{0\leq t \leq \sigma r^2}\sup_{x_0\in \R^3} \int_{B_r(x_0)}\frac {|u|^2} 2 \,dx\,dt + \sup_{x_0\in \R^3}\int_0^{\sigma r^2}\int_{B_r(x_0)} |\nabla u|^2\,dx\,dt <CA_0(r) ,
\end{equation}
where
\[
A_0(r)=rN^0_r= \sup_{x_0\in \R^3} \int_{B_r(x_0)} |u_0|^2 \,dx,
\] 
and
\begin{equation}\label{def.sigma}
\si=\sigma(r) =c_0\, \min\big\{(N^0_r)^{-2} , 1  \big\},
\end{equation}
for a small universal constant $c_0>0$.  Additionally, local energy solutions are mild \cite{BT7}.

The pressure expansion \eqref{eq:pressure.dec} is used when there is insufficient decay to define Calderon-Zygmund operators in the standard fashion. In the present paper this issue does not come up,  but we include the expansion in our definition to be consistent with the existing literature.\footnote{In fact, the pressure can be very simply defined for DSS solutions even in the absence of decay by re-scaling the far-field information to the compact cylinder $B(0,1)\times (0,1]$ \cite{BT3}.}

Local energy solutions can also be defined for generalizations of the Navier-Stokes equations, the only difference being that \eqref{ineq:CKN-LEI} needs to be modified appropriately. This comes up in Section \ref{sec:Besov}.

\subsection{Weak Besov space solutions}\label{sec:weakBesovSolutions}

Since the introduction of local energy solutions,  a class of weak solutions  has been developed by Seregin and \v Sver\' ak which achieves many of the same things as local energy solutions but is more tailored to initial data in critical classes \cite{SeSv,BaSeSv,AB}. The original paper of Seregin and \v Sver\' ak dealt with data in $L^2$ \cite{SeSv}. Barker, Seregin and \v Sver\' ak then extended the construction to $L^{3,\I}$ \cite{BaSeSv}.  Albritton and Barker addressed the Besov space case $\Bp$ for $p\in(3,\I)$ \cite{AB}.  In the case of the critical spaces $\Bp$, we do not have $\Bp\subset L^2_\uloc$---see an example in \cite{BT3}---and, therefore, when studying DSS solutions with $\Bp$ data, we cannot use local energy solutions. In order to generalize Theorem \ref{thrm:L3decay} to Besov spaces, which we do in Section \ref{sec:Besov}, we therefore work with the $\Bp$-weak solutions from \cite{AB} instead of local energy solutions.

\begin{definition}[Weak Besov Solutions]\label{defn:BesovSoln} Let $T>0$, $u_0 \in \Bp$ be a divergence-free vector field where $p\in (3,\I)$. We say that a distributional vector field on $\R^4_+$ is a \textit{weak Besov solution}, also written as ``$\Bp$-weak solution,'' to the Navier-Stokes equations with initial data $u_0$ if there exists an integer $k\ge 0$ such that the following conditions  are satisfied
\begin{enumerate}
\item there exists a pressure $q \in L^{3/2}_\loc(\R^4_+)$ such that $u$ satisfies the Navier-Stokes equations in the sense of distributions
\item $u$ may be decomposed as $u=v+P_k(u_0)$ for $v\in L^\I_T L^2 \cap L^2_T \dot H^1$ and $v(\cdot, t)$ is weakly $L^2$ continuous in time and converges to $0$ in $L^2$ as $t\to 0^+$.
\item $(u,q)$ satisfy the local energy inequality for all $t$ and every non-negative test function $\phi \in C^\I_0(\R^4_+)$
\EQ{
&\int_{\R^3} \phi(x,t) |u(x,t)|^2 \,dx + 2 \int_0^t \int_{\R^3} \phi |\nb u|^2 \,dx\,dt' \\
&\le \int_0^t \int_{\R^3} |u|^2 (\pd_t \phi + \De \phi)+ (|u|^2+2q)(u\cdot\nb \phi)\,dx\,dt'.}
\end{enumerate}
\end{definition}

A useful property of $\Bp$-weak solutions is that energy methods can be applied to $v=u-P_k(u_0)$. Based on this one has a decay property as $t\to 0$,
\[
\sup_{0<s<t} \| v\|_{L^2}^2(t)+\int_0^t \| v\|_{\dot H^1}^2\,ds \lesssim_{u_0} t^{1/2}.
\]

\section{Splittings of DSS data and Picard iterates} \label{sec.approximations}

An important fact used in this paper is that  discretely self-similar vector fields can often be approximated by elements of  classes which are not dense in the ambient space.    In this section, we investigate this theme further to develop tools to analyze far-field regularity and spatial decay of DSS solutions. 

Our first lemma is an approximation property for initial data in $L^{3,\I}\cap DSS$.  This type of result  is similar to \cite[Lemma 2.2]{BT3}.

Let $A_k = \{x:\la^k \leq |x|<\la^{k+1}  \}$ and $A_k^* = \{x:\la^{k-1} \leq |x|<\la^{k+2}  \}$. 

\begin{lemma}\label{lemma:L3wDataDecomp} 
Assume $u_0\in L^{q}_\loc(\R^3\setminus \{0\})$ for some $q\in[3,\I)$, is divergence free, and is $\la$-DSS for a fixed $\la>1$. Let $\e>0$ be given.  Then, there exist divergence free, DSS  vector fields $a_0$ and $b_0$ so that $u_0=a_0 + b_0$  and \[
\| a_0\|_{L^{3,\I}},\|a_0\|_{L^q(A_0)}<\e,\qquad |b_0(x)|\leq \frac {C_i(\e,u_0)} {|x|}. 
\]
The same conclusion follows with the divergence free property removed throughout.
\end{lemma} 


\begin{proof}
 We have $u_0|_{A_0}\in L^3(A_0)$ by Lebesgue space embeddings over domains of finite measure. By density of $C_c^\I(A_0)$ in $L^q(A_0)$ and $L^3(A_0)$, for any $\e'>0$, there exists $\td b_0\in C_c^\I(A_0)$ so that $\|   u_0 -\td b_0\|_{L^3(A_0)}\leq \e'$ and $\|   u_0 -\td b_0\|_{L^q(A_0)}\leq \e'$. Define $\td b_0$ globally by extending it via DSS scaling to all of $\R^3\setminus \{0\}$. Let $\td a_0 = u_0-\td b_0$. By \cite[(3.5)]{BT1} we have 
\[
\|\td a_0\|_{L^{3,\I}}\lesssim_\la  \| \td a_0\|_{L^3(A_0)} \leq \e'.
\]
Furthermore, 
\[
|\td b_0| (x)\lesssim_{\e,u_0} |x|^{-1},
\]
and $\td b_0\in C^\al_\loc(\R^3\setminus \{0\})$ for any $0<\al<1$. 

Let $a_0 = \mathbb P \td a_0$ and $b_0 = \mathbb P \td b_0$. Then, $a_0$ and $b_0$ are divergence free. Additionally,
\[
\| a_0 \|_{L^{3,\I}}\lesssim_{CZ} \| \td a_0 \|_{L^{3,\I}} \lesssim_{\la} \e',
\]
where the first suppressed constant depends on the constant from the Calderon-Zygmund theory. Choosing $\e'$ small in terms of this constant and $\la$ ensure that 
\[
\|a_0\|_{L^{3,\I}}<\e.
\]
We also show $\|a_0\|_{L^q(A_0)} < \e$. We have 
\EQN{
\|a_0\|_{L^q(A_0)} &\leq C \|a_0\|_{L^q(A_0^*)} +C |A_0|^{1/q} \bigg\| \int_{y\notin A_0^*}  \frac 1 {|x-y|^3}	    \td a_0  (y)	\bigg\|_{L^\I(A_0)}.
}
We bound the second term above in two cases. First note that by \cite[Lemma 6.1]{BT8},
which implies that for any measurable set $E$,
\[
\int_E |f| \lesssim \| f \|_{L^{3,\I}}|E|^{1-\frac 1 3},
\]
we have
\[
|A_0|^{1/q} \bigg\| \int_{|y|\leq \la^{-1}} \frac 1 {|x-y|^3} \td a_0(y) \,dy \bigg\|_{L^\I(A_0)}\lesssim_\la \| \td a_0 \|_{L^{3,\I}} \leq \e'. 
\]
On the other hand, by re-scaling $\td a_0$, we can pass from integrals over $A_k$ to integrals over $A_0$  to justify the following estimate
\EQ{\label{ineq.czType}
|A_0|^{1/q} \bigg\| \int_{|y|\geq  \la^{2}} \frac 1 {|x-y|^3} \td a_0(y) \,dy \bigg\|_{L^\I(A_0)}\lesssim_\la \sum_{k=2}^\I \la^{-k} \|\td a_0\|_{L^1(A_0)}\lesssim \|\td a_0\|_{L^3(A_0)}\leq \e'.
}
Hence by taking $\e'$ small as determined by $\la$ we obtain the desired conclusion.

We claim that $b_0|_{{A_0}}$ is bounded. Since $\mathbb P$ preserves discrete self-similarity \cite[p.~61]{BT3},  $b_0$ is DSS. Hence, boundedness on $A_0$  implies 
\[
|b_0(x)|\lesssim |x|^{-1},
\] and completes the proof. To prove the claim, fix $x\in {A_0}$. We have by the definition of  $\mathbb P$ that
\EQN{
| b_0 |(x) &\leq C \bigg[ p.v. \int_{ |x-y|< \frac 1 2} \frac 1 {|x-y|^3} |\td b_0|(y)\,dy + p.v.\int_{|x-y|\geq \frac 1 2} \frac 1 {|x-y|^3} |\td b_0|(y)\,dy\bigg]
\\&\lesssim  \| \td b_0\|_{C^\al(B_{\frac 1 2}(x)  )} +  \| \td b_0 \|_{L^{3,\I}},
}
where the first term comes from the H\"older regularity of $\td b_0$ away from $x=0$ and the last term can be deduced by arguing as in \eqref{ineq.czType}.
\end{proof} 

The same approximation result for $\Bp \cap DSS$ is a corollary of Lemma \ref{lemma:L3wDataDecomp} and \cite[Lemma 2.2]{BT3}.  Note that this is a refinement of \cite[Lemma 2.2]{BT3}. It is possible to revise the proof in \cite{BT3} to get the desired result without   Lemma \ref{lemma:L3wDataDecomp}, but doing so would be less efficient here.

\begin{lemma}\label{lemma:BpDataDecomp} 
Assume $u_0\in \Bp$, is divergence free and is $\la$-DSS for a fixed $\la>1$. Let $\e>0$ be given.  Then, there exist divergence free, DSS vector fields $a_0$ and $b_0$ so that $u_0=a_0 + b_0$  and \[
\| a_0\|_{\Bp}<\e,\qquad |b_0(x)|\leq \frac {C_{ii}(\e,u_0)} {|x|}. 
\]The same conclusion follows with the divergence free property removed throughout.
\end{lemma}

\begin{proof}
We have by \cite[Lemma 2.2]{BT3} that, for any $\e'>0$,  there exists $\bar a_0$ and $\bar b_0$  which are DSS and divergence free so that 
\[
\|\bar a_0\|_{\Bp} <\e',
\]
and  $\bar b_0\in L^{3,\I}$. Then, applying Lemma \ref{lemma:L3wDataDecomp} with $q=3$ to $\bar b_0$ allows us to write $\bar b_0 =\td a_0 +\td b_0$ where $\|\td a_0\|_{L^{3,\I}}<\e'$, $|\td b_0|(x)\lesssim |x|^{-1}$ and both are DSS and divergence free. Let $a_0 = \bar a_0  + \td a_0$ and let $b_0 = \td b_0$. We have
\[
\|a_0\|_{\Bp} \leq 	\| \bar a_0\|_{\Bp} + C \| \td a_0\|_{L^{3,\I}} \lesssim \e',
\] 
where the constant comes from the embedding $L^{3,\I}\subset \Bp$. We may therefore choose $\e'$ small to ensure $\|a_0\|_{\Bp} <\e$.
\end{proof}

The preceding approximation properties are useful because they imply that DSS elements of $L^\I_\loc(\R^3\setminus \{0\})$ are, in some sense, dense in larger, critical classes of DSS fields. This means that the algebraic decay properties of the larger classes are, modulo an error which can be made small, the same as $L^\I_\loc(\R^3\setminus \{0\})\cap DSS$, i.e. $|x|^{-1}$.

We now extend the approximation property of the initial data to linear evolution equations. To begin, we need a decay estimate for DSS solutions to the heat equation.
\begin{lemma}[\cite{BP1}, Lemma $3.1$]\label{lemma:heat.equation.pointwise.decay}Assume $u_0\in L^p_\loc(\R^3\setminus \{0\})$ where $p\in (3,\I]$ and is DSS. Then,
\[
\sup_{t\in [1,\la^2]} \| e^{t\Delta} u_0 \|_{L^\I(B_R^c)} \lesssim \| u_0\|_{L^p(A_1)} R^{ \frac 3 p -1 }.
\]
\end{lemma}
 Note that this conclusion was originally  discussed without proof by Tsai and the first author in \cite{BT1}. When $p=3$, decay can still be proven but there is no universal algebraic decay rate   as demonstrated in \cite{BT1}.

The next lemma establishes a decomposition for the solution to the heat equation with initial data $u_0$.
\begin{lemma}\label{lemma:decomp.heat}
Assume $u_0\in L^q_\loc(\R^3\setminus \{0\})$ for some $q\in (3,\I)$ and is DSS and divergence free. For any $\e>0$, there exist divergence free, DSS $P_{0,1}$ and $P_{0,2}$ so that $e^{t\Delta}u_0= P_{0,1}+P_{0,2}$ and 
\[
|P_{0,1}(x,t)| < \frac \e {\sqrt t^{\frac 3 q} (|x|+\sqrt t)^{1-\frac 3 q} };\qquad |P_{0,2}(x,t)|   \leq \frac   {C_{iii} (\e,u_0)  } {|x|+\sqrt t},
\]
for a constant $C_{iii} (\e,u_0)$ which blows up as $\e\to 0$. Consequently,
\EQ{\label{lim:vanishing.heat}
\lim_{r\to \I} \sup_{|x|\geq r\sqrt t} |e^{t\Delta}u_0(x)|\sqrt t^{\frac 3 q} (|x|+\sqrt t)^{1-\frac 3 q} = 0.
}
\end{lemma} 
\begin{proof}
We use Lemma \ref{lemma:L3wDataDecomp} to split $u_0$ into $a_0+b_0$. We next apply Lemma \ref{lemma:heat.equation.pointwise.decay} to each term. Since Lemma \ref{lemma:L3wDataDecomp} allows us to make $\| a_0\|_{L^q(A_1)}$  as small as we like, we obtain the pointwise estimates for $P_{0,1}=e^{t\Delta}a_0$ and $P_{0,2}=e^{t\Delta}b_0$ when $t\in [1,\la^2]$. DSS scaling extends these estimates to all of $\R^3\times (0,\I)$.

For the last claim, given $\e' >0$, we need to show 
\[ 
 |e^{t\Delta}u_0(x,t)|\sqrt t^{\frac 3 q} (|x|+\sqrt t)^{1-\frac 3 q}  <\e',
\]
for $|x|\geq  R(\e') \sqrt t$ where $R$ is chosen based on $\e'$. Take $\e= \e'/2$ above and decompose $e^{t\Delta}u_0 = P_{0,1} +P_{0,2}$ where
\[
|P_{0,1}(x,t)|\leq \frac {\e'/2} {\sqrt t^{\frac 3 q} (|x|+\sqrt t)^{1-\frac 3 q} }\text{ and } |P_{0,2}(x,t)|\lesssim \frac {C_{iii}(\e'/2,u_0)} {|x|+\sqrt t}.
\]
Assuming $|x|\geq  R(\e') \sqrt t$ we have
\EQ{
 |e^{t\Delta}u_0(x,t)|\sqrt t^{\frac 3 q} (|x|+\sqrt t)^{1-\frac 3 q} &\leq \frac {\e'}2 + \frac { C_{iii}(\e',u_0)}{(\frac{|x|}{\sqrt{t}}+1)^{\frac 3 q}}
\\&\leq \frac {\e'}2 + \frac{C_{iii}(\e',u_0)}{(R+1)^{\frac 3 q}}.
}
By taking $R$ large based on $\e'$ we can control the preceding terms by $\e'$. Therefore
\EQ{\lim_{r\to \I} \sup_{|x|\geq r\sqrt t} |e^{t\Delta}u_0(x)|\sqrt t^{\frac 3 q} (|x|+\sqrt t)^{1-\frac 3 q} = 0.}

\end{proof}

\begin{proposition}\label{proposition:PIdecomp2} Fix $q\in (3,\I)$. Suppose $u_0\in L^q_\loc(\R^3 \setminus \{0\})$ is divergence free and DSS. Given $\e>0$ and $k\in \N$, there exist constants $C_{iii}$, $C_{iiii}$ and $C'$  and vector fields
  $P_{k,1}$ and $P_{k,2}$ so that  
\EQ{P_k(u_0)= P_{k,1}+P_{k,2}}
\EQ{\label{eq:ak}
|P_{k,1}(x,t)|\leq \frac{(2-2^{-k})\varepsilon}{\sqrt{t}^{\frac 3 q}(|x|+\sqrt t)^{1-\frac 3 q}},
}
and,
\EQ{\label{eq:bk}
|P_{k,2}(x,t)|\lesssim \frac {C_{iii}(\e,u_0,k)} {|x|+\sqrt t}.
}
Furthermore, there exist $P_{k,1}'$ and $P_{k,2}'$  so that the difference $P_k-P_{k-1} =P_{k,1}'+P_{k,2}'$ where $P_{k,1}'$ and $P_{k,2}'$ satisfy
\EQ{\label{ineq:higherPk1}|P_{k,1}'(x,t)| \leq  \frac {C_{iiii}(k)\varepsilon} {\sqrt{t}^{1-a_k} (|x|+\sqrt{t})^{a_k}}}
 and 
\EQ{\label{ineq:higherPk2}|P_{k,2}'(x,t)| \leq \frac {C'(\varepsilon,u_0,k)\sqrt{t}^{b_k}} { (|x|+\sqrt{t})^{b_k+1}},}
 where $a_k=\min\{(k+1)(1-\frac 3q),4\}$  and $b_1=1$, $b_2=2$, and $b_{k+1}=\min\{   b_k+1-3/q  , 4  \}$.
\end{proposition}  

We do not require $P_{k,i}$ to be divergence free but it seems possible to enforce this condition. In our applications, we will need the initial data splittings to be divergence free but it does not appear necessary for this to extend to the Picard iterate splittings.

\begin{proof}
We have the above for $P_0$ by \cref{lemma:decomp.heat}. 
Now assume the above holds for $P_k$. Then,
\EQ{P_{k+1} -P_0= B(P_k,P_k)
=B(P_{k,1}+P_{k,2},P_{k,1}+P_{k,2})\\
=-(A_{k+1}+B_{k+1}+C_{k+1})
}
where
 \EQ{A_{k+1} &:=  \int_0^t e^{(t-s)\De}\mathbb{P} \nabla \cdot \big ( P_{k,1} \otimes P_{k,1} \big) \, ds,\\
 B_{k+1} &:= \int_0^t e^{(t-s)\De}\mathbb{P} \nabla \cdot \big ( P_{k,2}\otimes P_{k,2} \big) \, ds,\\
 C_{k+1} &:=  \int_0^t e^{(t-s)\De}\mathbb{P} \nabla \cdot \big ( (P_{k,1}\otimes P_{k,2}) + (P_{k,2}\otimes P_{k,1}) \big) \, ds.}
For the first term, $A_{k+1}$, 
\EQ{
A_{k+1}
&\lesssim \int_0^t \int \frac 1 {(|x-y|+\sqrt{t-s})^4} \left( \frac \varepsilon {\sqrt{s}^{\frac 3 q} (|y|+\sqrt s)^{1-\frac 3 q}}\right)^2 \, dy\, ds\\
&\lesssim \frac {\varepsilon^2}{\sqrt{t}^{\frac 6 q -1} (|x|+\sqrt t)^{2-\frac 6 q}},
}
by \cref{lemma.intbound1}. By the same lemma,
\EQ{
C_{k+1} 
&\lesssim \int_0^t \int \frac 1 {(|x-y|+\sqrt{t-s})^4} \frac \varepsilon {\sqrt{s}^{\frac 3 q} (|y|+\sqrt s)^{1-\frac 3 q}} \frac{ C_{iii}(\varepsilon, u_0,k)} {|y|+\sqrt s}\, dy\, ds\\
&\lesssim \frac {\varepsilon C_{iii}(\varepsilon, u_0,k) }{\sqrt{t}^{\frac 3 q -1} (|x|+\sqrt t)^{2-\frac 3 q}}.
}
Lastly, for $B_{k+1}$, we use \eqref{ineq:Tsai.integral} to find
\EQ{
B_{k+1} 
&\lesssim \int_0^t \int \frac 1 {(|x-y|+\sqrt{t-s})^4} \left(\frac{ C_{iii}(\varepsilon, u_0,k)  } {|y|+\sqrt s}\right)^2\, dy\, ds\\
&\lesssim \frac { C_{iii}(\varepsilon, u_0,k)^2 }{(|x|+\sqrt t)^{2}}.
}
Note that each term decays faster than $\sqrt{t}^{-\frac 3 q}(|x|+\sqrt t)^{-1+\frac 3 q}$. 

As a brief aside we  emphasize   that $P_k-P_0$ has \textit{more} decay than either $P_k$ or $P_0$ and can be written as a sum of terms bounded by multiples of 
\EQ{ \frac \varepsilon {\sqrt{t}^{\frac 6q -1}(|x|+\sqrt t )^{2-\frac 6q}} \text{ and } \frac{C(\varepsilon,u_0,k)\sqrt t}{(|x|+\sqrt t )^2}.}
When $k=1$ this will be used as the base case for an inductive argument later in this proof. 

Let $\td \chi(x)$ be a smooth cut-off function supported in $B(0,2)$ and equal to $1$ in $B(0,1)$. Let $\chi_R(x,t)= \chi ( x/(\sqrt t R))$. By taking $R$ large, we have that 
\[
(A_{k+1} + C_{k+1} )(1-\chi_R(x,t))\leq \frac {\e (2^{-k}- {2^{-(k+1)})}}{\sqrt{t}^{\frac 3 q}(|x|+\sqrt t)^{1-\frac 3 q}},
\]
while $(A_{k+1} + C_{k+1} ) \chi_R(x,t)+B_{k+1} $ is bounded by
\[
\frac {C_{iii}(\e,u_0,k+1)} {|x|+\sqrt t},
\]
for a suitable choice of $C_{iii}(\e,u_0,k+1)$---note that we used DSS scaling and the fact that $A_{k+1}$ and $C_{k+1}\in L^\I_\loc(\R^3\times [1,\la^2])$.

\bigskip Now assume  that $P_k-P_{k-1} =P_{k,1}'+P_{k,2}'$ where $P_{k,1}'$ and $P_{k,2}'$ satisfy
\EQ{|P_{k,1}'(x,t)| \leq \frac {C_{iiii}(k) \varepsilon }{\sqrt{t}^{1-a_k} (|x|+\sqrt{t})^{a_k}}}
 and 
\EQ{|P_{k,2}'(x,t)| \lesssim \frac {C'(\varepsilon,u_0,k)\sqrt{t}^{b_k}} { (|x|+\sqrt{t})^{b_k+1}},}
 where $a_k=\min\{(k+1)(1-\frac 3q),4\}$ and $b_1=1$, $b_2=2$, and $b_{k+1}=\min\{  b_k+1-3/q  , 4  \}$. This assertion holds when $k=1$ by the first part of this proof. 
 Then,
\EQ{P_{k+1} -P_k&= B(P_{k},P_{k})-B(P_{k-1},P_{k-1})\\&=
B(P_{k}-P_{k-1},P_{k})+B(P_{k-1},P_k-P_{k-1})\\&=B(P_{k,1}'+P_{k,2}',P_{k})+B(P_{k-1},P_{k,1}'+P_{k,2}')\\
&=:P_{k+1,1}'+P_{k+1,2}'
}where
 \EQ{P_{k+1,1}'&:=  -\int_0^t e^{(t-s)\De}\mathbb{P} \nabla \cdot \big ( P_{k,1}' \otimes P_{k} +P_{k-1} \otimes  P_{k,1}' \big) \, ds,\\
 P_{k+1,2}' 
 &:= -\int_0^t e^{(t-s)\De}\mathbb{P} \nabla \cdot \big ( P_{k,2}'\otimes P_{k}  +  P_{k-1}\otimes P_{k,2}'\big) \, ds.}
For the first term,
\EQ{
P_{k+1,1}'
&\leq  C(u_0,k) \int_0^t \int \frac 1 {(|x-y|+\sqrt{t-s})^4} \frac {C_{iiii}(k)\varepsilon \sqrt{s}^{a_k-1}}{ (|y|+\sqrt s)^{a_k}} \frac {\sqrt s^{-\frac 3 q}} {(|y|+\sqrt s)^{1-\frac 3 q}} \, dy\, ds\\
&\leq C(u_0,k) \frac {C_{iiii}(k) \varepsilon  \sqrt{t}^{a_{k+1}-1}}{ (|x|+\sqrt t)^{a_{k+1}}} 
=: \frac {C_{iiii}(k+1) \varepsilon  \sqrt{t}^{a_{k+1}-1}}{ (|x|+\sqrt t)^{a_{k+1}}},
}
by \cref{lemma.intbound1} and where the upper bound for $P_k$ comes from  combining \cite[Theorem 1.1]{BP1} and \cite[Theorem 1.2]{BP1}---note that the   dependence of the otherwise universal constant on $k$ and $u_0$ comes from  \cite{BP1}.
By the same lemma,
\EQ{
P_{k+1,2}'
&\lesssim_{u_0,k} \int_0^t \int \frac 1 {(|x-y|+\sqrt{t-s})^4} \frac{ C'(\varepsilon, u_0)\sqrt{s}^{b_k}} {(|y|+\sqrt s)^{b_k+1}} \frac {1} {\sqrt s^{\frac 3 q}(|y|+\sqrt s)^{1-\frac 3 q}}\, dy\, ds.
\\
&\lesssim \frac { C'(\varepsilon, u_0,k) \sqrt{t}^{b_{k}+1-3/q}}{(|x|+\sqrt t)^{b_k+2-3/q}}
=:\frac { C'(\varepsilon, u_0,k+1) \sqrt{t}^{b_{k+1} }}{(|x|+\sqrt t)^{b_{k+1}+1}},
}
where $b_{k+1} = b_k+1-3/q$.

\end{proof}

\section{Asymptotics of DSS Navier-Stokes flows}

\subsection{Vanishing algebraic decay}
We may now prove \cref{thrm:Lpdecay1} using \cref{proposition:PIdecomp2} and \cite[Theorem $1.2$]{BP1}.
\begin{proof}[Proof of \cref{thrm:Lpdecay1}]
 By \cref{proposition:PIdecomp2}, for any $\e>0$ we can write $P_k = P_{k,1}+P_{k,2}$ such that 
\EQ{|P_{k,1}(x,t)| \le \frac {\varepsilon (2-2^{-k})} {\sqrt{t}^{\frac 3 q} (|x|+\sqrt{t})^{1-\frac 3 q}}, \text{ and } |P_{k,2}(x,t)|\le \frac{C_{iii}(\varepsilon,u_0,k)}{|x|+\sqrt{t}}.}
Write $u = u-P_k + P_k = \left(u-P_k+P_{k,2}\right)+P_{k,1}$, $k$ to be specified momentarily.
Next, by \cite[Theorem $1.2$]{BP1}, we know that, 
\EQ{|u-P_k|(x,t)\lesssim_{u_0,\la,k} \frac 1 {\sqrt{t}^{1- (k+2)(1-\frac 3 q)} (|x|+\sqrt{t})^{(k+2)(1-\frac 3 q)}},}
in the region $|x|\ge R_0 \sqrt t$, for $(k+2)(1-\frac 3 q) \le 4$. Clearly, for $ k\ge \frac q {q-3}-2$, \EQ{|u-P_k|(x,t)\lesssim_{u_0,\la,k}\frac {1} {|x|+\sqrt{t}},}
in the same region.
Therefore the first part of the theorem is proven, with $a(x,t)=P_{k,1}(x,t)$ and $b(x,t)=u-P_k+P_{k,2}$.

We now prove the various little-o properties asserted in the theorem. Observe that 
\EQ{|u(x,t)| \sqrt{t}^{\frac{3}{q}}(|x|+\sqrt{t})^{1-\frac{3}{q}} \le& |a(x,t)| \sqrt{t}^{\frac{3}{q}}(|x|+\sqrt{t})^{1-\frac{3}{q}}\\
&+ |b(x,t)| \sqrt{t}^{\frac{3}{q}}(|x|+\sqrt{t})^{1-\frac{3}{q}}\\
\leq& \varepsilon + C(\varepsilon,u_0,k,\la) \sqrt{t}^{\frac{3}{q}}(|x|+\sqrt{t})^{-\frac{3}{q}},}
where we are taking $|x|\geq R_0\sqrt t$.
Since we are further interested in $|x|\ge r\sqrt{t}$, we have 
\EQ{ C(\varepsilon,u_0,k,\la)\sqrt{t}^{\frac{3}{q}}(|x|+\sqrt{t})^{-\frac{3}{q}} =  C(\varepsilon,u_0,k,\la) (\frac{|x|}{\sqrt{t}}+1)^{-\frac{3}{q}}\le \frac{C(\varepsilon,u_0,k,\la)}{(r+1)^{3/q}}.} 
For a given $\e'>0$, and taking $\e=\e'/2$, we may now choose $r$ large to ensure the preceding quantity is bounded by $\e'/2$.
Therefore,
\EQ{\lim_{r\to \I}\sup_{|x|\ge r \sqrt t} |u(x,t)| \sqrt{t}^{\frac{3}{q}}(|x|+\sqrt{t})^{1-\frac{3}{q}}=0.}

Moving on to $u-P_k$ where $k\in \N_0$, we have 
\[
u-P_k = u-P_{k+1} + P_{k+1}-P_k. 
\]
By \cite{BP1}, 
\[
|u-P_{k+1}| \lesssim_{k,\la,u_0} \frac 1 {\sqrt t^{1-a_{k+1}} (|x|+\sqrt t )^{a_{k+1}}},  
\]
while by Proposition \ref{proposition:PIdecomp2},
\[
|P_{k+1}-P_k | \leq \frac {C_{iiii}(k) \e} {\sqrt t^{1-{a_k}}(|x|+\sqrt t)^{a_k} } + \frac {C' \sqrt t^{b_k}} { (|x|+\sqrt t)^{b_k+1}}.
\]
Observe that for $a_k<4$ we have $b_k > a_k$ and $a_{k+1}> a_k$. For a given $\e'>0$ we can therefore choose $\e =   \e' /(2 C_{iiii}(k))$ and then choose $r$ large to make the remaining term small, proving the claim.

\end{proof}

\subsection{Almost algebraic decay for $L^{3,\I}$ data}
We begin with a simple lemma on strong solutions to the perturbed Navier-Stokes equations.
\begin{lemma}\label{lemma:Linfty} Fix $p\in (3,\I]$.
If $u_0\in L^p$ is divergence free and $a$ is a divergence free vector field in the Kato class $\mathcal K_\I$ satisfying
\[
\| a\|_{\mathcal K_\I}<\e,
\] 
for a universal constant $\e$, then, letting $T = T(u_0) \sim \|u_0\|_{L^p}^{-\frac 1 2 (1-\frac 3 p)}$, there exists a unique strong solution to 
\EQ{\label{eq:pns}
\partial_t u -\nu \Delta u +\mathbb P\nb\cdot ( u\otimes u +a\otimes  u+u\otimes a) = 0;\quad \nb \cdot u = 0,
}
which is mild and satisfies $u\in C((0,T];L^p)$ with 
\[
\sup_{0<t<T} \| u(t)\|_{L^p}\leq 2\|u_0\|_{L^p},
\]
and, if $p<\I$,
\[
\lim_{t\to 0} \| e^{t\Delta}u_0 - u(t)\|_{L^p}=0.
\]
If $u_0\in L^p\cap L^q$ for $p,q\in (3,\I]$, then the solutions generated above are the same.
\end{lemma}

This follows by modifying a standard fixed point argument \cite{FJR,GIM,Ku,Kato}.
Since the above solution is mild, it is also a local energy solution due to the arguments in \cite{FDLR,BT7}. 
To prove this we will use the following fixed point theorem.
\begin{proposition}\label{prop.fp}
If $E$ is a Banach space and $B: E\times E \to E$ is a bounded bilinear transform satisfying \EQ{\label{contraction}
\| B(e,f)\|_{E}\leq C_B \| e\|_E\|f\|_E,
}
and if $\|e_0\|_{E}\le \de\le (4C_B)^{-1}$, and  $a$ is given and satisfies,
\EQ{\label{contraction2}
\| B(e,a)\|_{E} + \| B(a,e)\|_{E}\leq \frac 1 {8} \|e\|_E,
}
then the equation $e = e_0 - B(e,e)-B(U,e)-B(e,U)$  has a solution with $\| e\|_{E}\leq 3 \delta/2$ and this solution is unique in $\overline B(0,3\delta/2)$.
\end{proposition}

The proof of this is a simple variation on the usual Picard iteration argument.

\begin{proof}[Proof of Lemma \ref{lemma:Linfty}]
This is a variation of a standard approach using classical ideas---\cite{Kato,Ku,GIM}.  The only difference compared to \cite{Ku} is estimating the terms involving $a$. Considering the case when $p=\I$, for example, the difference is visible in the following estimate
\EQ{
\int_0^t \frac 1 {(t-s)^{\frac 1 2}} \| a \otimes e +e\otimes a\|_{L^\I }(s) \,ds&\leq \|a\|_{\mathcal K_\I} \int_0^t \frac 1 {(t-s)^{\frac 1 2} s^{\frac 1 2}} \| e \|_{L^\I (0,T;L^\I)}
\\&\leq C \|a\|_{\mathcal K_\I}\| e \|_{L^\I (0,T;L^\I)}.
}
If $\e$ is sufficiently small and taking $\|a\|_{\mathcal K_\I}<\e$ we can close the estimate and apply Proposition \ref{prop.fp}.
\end{proof}

Our decomposition will follow from a local smoothing result of Barker and Prange \cite[Theorem 1]{BaPr2020} which we will recall below.  Related results appear in \cite{KMT,Kwon} and are based on the foundational work \cite{JS}.  The crucial part of the result for us involves a localized solution which we first introduce. Given some divergence free $u_0$ we require $a_0$ to satisfy the following properties: $\supp a_0 = B_2(0)$; $\nb \cdot a_0 = 0 ; a_0|_{B_1(0)} =u_0|_{B_1(0)} $. The function $a_0$ can be constructed using the Bogovskii map. If $u_0\in L^3(B_2(0)$ and is sufficiently small, then there exists  global smooth solution evolving from $a_0$ by \cite{Kato}---we denote this solution, which appears in the following theorem, by $a$. 

\begin{theorem}[Barker \& Prange]\label{thrm.BPmain}Suppose $(u,p)$ is a local energy solution to the Navier-Stokes equations with data $u_0\in E^2$. There exists a universal constant $\gamma_{\text{univ}}>0$ so that if
\[
\|u_0\|_{L^3(B_2(0))}\leq \gamma_{\text{univ}},
\]then there exists $T=T(u_0)$ so that 
\EQ{\label{concl.BaPr2020}
u-a \in C^{0,\si}_{\text{par}} (\bar B_{1/3}\times [0,T]).
}
\end{theorem}
Note that the ball in the preceding theorem can be centered anywhere, not just at $x=0$.

We are now able to prove Theorem \ref{thrm:L3decay}, which is an application of Theorem \ref{thrm.BPmain}.

\begin{proof}[Proof of Theorem \ref{thrm:L3decay}]

To apply Theorem \ref{thrm.BPmain} to discretely self-similar solutions, we note that $\| u_0\|_{L^3(A_k)}=\|u_0\|_{A_1}$. Hence, there exists a large enough $\td k$ so that for every $x_0\in A_{\td k}$, $\|u_0\|_{L^3(B_2(x_0))}<\ga_{\text{univ}}$---this is because we're spreading the same $L^3$-mass over a larger volume. Let us fix a finite cover of $\bar A_{\td k}$ by a collection of balls $B_{1/3}(x_i)$ where $x_i\in A_{\td k}$. 

We can now generate a collection of functions $\{\td a_{x_i}\}$ and times $\{T_{x_i}\}$ so that \eqref{concl.BaPr2020} holds for each $i$. Let $N$ be the number of elements of our cover. For each $x_i$ we can further refine the corresponding initial datum $\td a_{x_i,0}$ into a sum of divergence free vector fields one of which is in $L^3$ and is smaller than $\e/(2N)$ and another which is in $L^\I$, label these $\td a_{x_i,0}^1$ and $\td a_{x_i,0}^2$ respectively. We use the Kato theory to solve for a strong global solution $\td a_{x_i}^1$ evolving from $\td a_{x_i,0}^1$ and use Lemma \ref{lemma:Linfty} to solve for a strong local solution $\td a_{x_i}^2$ which is bounded on some finite time-span $T_{x_i}'$. By uniqueness of Kato's solutions,  $\td a_{x_i}= \td a_{x_i}^1+\td a_{x_i}^2$. Let $\td b_{x_i}=u-\td a_{x_i}$. Let $T=\min\{ T_{x_i},T_{x_i}'\}$.

To summarize, we have that
\EQ{\label{conditionLIST}
\sup_{0<t<T;x_i} \| \td a_{x_i}^1\|_{L^3}(t)\leq \frac \e N;\quad  \sup_{0<t<T;x_i} \| \td a_{x_i}^2\|_{L^\I}(t)\leq C(u_0);\quad   \|\td b_{x_i} \|_{C^{0,\si}_{\text{par}}(\bar B_{1/3}(x_i)\times [0,T])}\leq C(u_0).
}
We now define $a$.
For each $x\in A_{\td k}$, define $\td a^1 $ to be $\td a_{x_i}^1$ where $i$ is the smallest number so that $x\in B_{1/3}(x_i)$. This unambiguously defines a vector field in $A_{\td k}  \times [0,T]$. Let $a=0$ on $A_{\td k}  \times (T,\I)$. Define $\td a^2 $ and $\td b$ analogously, but let $\td b=u$ on $A_{\td k}  \times (T,\I)$. Extend these to vector fields  $a_1$, $a_2$ and $b$ defined on $\R^3\setminus \{0\} \times [0,\I)$ using DSS scaling.  These satisfy the  set of conditions   in Theorem \ref{thrm:L3decay} involving the Herz space $\dot K^0_{3,\I}$. 

The proof of the result involving the Lorentz space $L^{3,\I}$  is still based on \cite{BaPr2020} but is more involved because it is not obvious that $a_1$ is small in $L^{3,\I}$.  
In the next section we state and prove  a version of Theorem \ref{thrm:L3decay} for Besov spaces. Because the proof of the $L^{3,\I}$ result is virtually identical to the more general Besov space theorem, we have opted to omit the details for the $L^{3,\I}$ conclusion here and instead include them when we prove the Besov  space version in the next section.

\end{proof}

\subsection{Almost algebraic decay with Besov space data}\label{sec:Besov}

In this section, we explore a generalization of Theorem \ref{thrm:L3decay} to Besov spaces.
We cannot use local energy solutions as our class of initial data  because there exist discretely self-similar vector fields in $\Bp\setminus L^2_\loc$ \cite{BT3}. Hence we work with the $\Bp$-weak solutions of Albritton and Barker \cite{AB} as defined in Section \ref{sec:weakBesovSolutions}. If we were proving the $L^{3,\I}$ case in full then we could either use the $L^{3,\I}$-weak solutions of Barker, Seregin and \v Sver\'ak \cite{BaSeSv} or we could use local energy solutions because $L^{3,\I}\subset L^2_\uloc$.

\begin{theorem}[Almost algebraic decay in $\Bp$]\label{thrm:Bpdecay}
Suppose $u_0$ is divergence free, $\la$-DSS and belongs to $\Bp$, $p>3$. Let $u$ be a $\Bp$-weak solution with data $u_0$. Then, for any $\e>0$ there exist $R_0>0$, $\si\in (0,1/2)$ and   DSS vector fields $a$, $b_1$ and $b_2$ so that \[
u(x,t) = a(x,t)+b_1(x,t)+b_2(x,t),
\]
for $|x|\ge R_0 \sqrt{t}$ and  
\[
\sup_{0<t<\I}\|a\|_{\Bp} <  \e;\qquad
|b_1(x,t)| \leq \frac   {C_{***} (\e,u_0) } {|x|+\sqrt t};\qquad  |b_2(x,t)|  \leq \frac   {C_{***} (\e,u_0,\si) \sqrt t^{2\si} } {(|x|+\sqrt t)^{1+2\si}},
\]
for a constant $C_{***}$.
\end{theorem}

Note that the local smoothing result of Barker and Prange, Theorem \ref{thrm.BPmain}, used the assumption $u_0\in L^2_\uloc$ when estimating the pressure at a critical juncture of their proof. For us, as mentioned above, we do not know that $u_0\in L^2_\uloc$. Hence the local energy bounds used in \cite{BaPr2020} to prove Theorem \ref{thrm.BPmain} are unavailable. In \cite{BaPr2020}, they additionally study versions of Theorem \ref{thrm.BPmain} where the data is locally small in either $L^{3,\I}$ or $\Bp$. In both cases, $u_0\in L^2_\uloc$ is still required for the same reason as the $L^3$ case. We therefore cannot use any of these results directly.  However, the local energy methods in the $L^{3,\I}$ and $\Bp$ cases are applied to a perturbed Navier-Stokes equation where the drift terms are in critical Kato classes. 
This insight proves crucial for us. 

To prove Theorem \ref{thrm:Bpdecay}, we will apply a double perturbation to obtain a  solution to a perturbed Navier-Stokes equations which \textit{does} have its initial data in $L^2_\uloc$ and has drift satisfying the same properties as were dealt with in \cite{BaPr2020}. Once we do this, the workhorse results in \cite{BaPr2020} can be applied to obtain the same conclusion about local smoothing.  

To unpack our ``double perturbation'' argument slightly, we recall how local smoothing is proven  in \cite{JS,BaPr2020}. The basic idea   is that, if the initial data $u_0$ is locally bounded on a ball $B(x,2)$, then $u_0$ is decomposed into $b_{0,1}$ and $b_{0,2}$ where $b_{0,1}$ is bounded globally and $b_{0,2}=0$ on $B(x,1)$. The Navier-Stokes equations are then solved for data $b_{0,1}$, which gives rise to a bounded solution $b_1$---this follows from the standard local well-posedness theory. We then consider the perturbed Navier-Stokes equations for $b_2=u-b_1$ ($u$ is a local energy solution with data $u_0$). Because $b_2|_{t=0} = 0$ on $B(x,1)$, $b_2$ can be extended backwards in time there. By considering a space-time cylinder $B(x,1)\times [-1+S_*,S_*]$ where $S_*$ is taken close enough to zero, it is possible to induce the smallness of
\EQ{\label{makeMeSmall}
\iint_{B(x,1)\times [-1+S_*,S_*]} |b_2|^3\,dx\,ds,
}
which, along with a similar bound for the pressure, implies $b_2$ is bounded in $B(x,1/4)\times [-1/4+S_*,S_*]$ by $\e$-regularity. Hence $u=b_1+b_2$ is bounded in $B(x,1/4)\times [0,S_*]$.  The estimates which are used to make \eqref{makeMeSmall} small are the local energy methods which require $u_0\in L^2_\uloc$.

We will modify this reasoning by decomposing $u_0$ into $a_0 +b_0$ where $a_0$ is small in $\Bp$ and $b_0\in L^\I_\loc(\R^3\setminus \{0\})\cap DSS$. Noting the last inclusion implies $b_0\in L^2_\uloc$, we then want to apply the same reasoning as in the preceding paragraph  to $b_0$. Compared to the above, $b$ solves a perturbed Navier-Stokes equations  with a small perturbation. In the above argument, $u$ corresponds to $b$ and $u$ solves the non-perturbed Navier-Stokes equations \eqref{eq:NS}. Hence, we need to revisit the details of the proof of local smoothing to incorporate a drift term which is small in $\Bp$ and $\mathcal K_\I$. This turns out to have been necessary in \cite{BaPr2020} where local smoothing was proven for $b_{0,1}$ small in $\Bp(B(x,2))$. While this is not exactly our setting, from the perspective of $b_{0,2}$, things ultimately look exactly the same for us compared to \cite{BaPr2020}. We therefore endeavor to work through our argument up to the point where things agree exactly with \cite{BaPr2020} and then outsource our reasoning to several supporting results in \cite{BaPr2020}. We presently recall these supporting results.

The next proposition concerns the existence of strong solution for small data in $\Bp$ and is taken directly from \cite[Proposition 23]{BaPr2020}. 
\begin{proposition}\label{prop:BaPr2020no4}
Let $S>0$ and $p\in (3,\I)$. There exist real numbers $\gamma(p)>0$ and $K''(p)>0$ such that the following holds true. For all divergence free $u_{0,a}\in\Bp(\R^3)$, for which
\[\sup_{0<t<S} t^{\frac 12 (1-\frac 3p)}\|e^{t\De}u_{0,a}\|_{L^p} \le \ga(p),\]
there exists a smooth mild solution $a\in C_{w^*}([0,S);\Bp)\cap L^\I((0,S);\Bp)$ of \eqref{eq:NS} such that $a(\cdot,0)=u_{0,a}$ and 
\[\sup_{t\in(0,S)} \|a(\cdot,t)\|_{\Bp} + t^{\frac 12(1-\frac 3p)}\|a(\cdot,t)\|_{L^p}+ t^{\frac 12}\|a(\cdot,t)\|_{L^\I} \le \sup_{t\in(0,S)}t^{\frac 12(1-\frac 3p)}\|e^{t\De}u_{0,a}\|_{L^p}.\]
The mild solution is unique in the class of solutions with sufficiently small
\[\sup_{t\in(0,S)}t^{\frac 12(1-\frac 3p)}\|e^{t\De}u_{0,a}\|_{L^p}\]
norm.
\end{proposition}

The next theorem is exactly \cite[Theorem 6]{BaPr2020}, which is a modification of \cite[Theorem 4]{BaPr2020} to accommodate Besov space data.  
\begin{theorem}\label{prop:BaPr2020no5} Let $t_0\in[-1,0]$ and $\eta \in (0,1)$ be fixed. For all $\ga\in(0,3)$ there exists $C_*(\ga)>0$, for all $E>0$, there exists $\varepsilon_*(\delta,\eta,E)>0$, for all $a$ such that
\EQ{\label{eq:BP5hypothesis0}
\sup_{s\in(-1,0)}|s-t_0|^{\frac 12} \|a(\cdot,s)\|_{L^\I(B_1(0))}<\I,}
and all local suitable solutions $v$ to \eqref{eq:pns} in $Q_1(0,0)$ such that for all $s\in(-1,0)$ 
\EQ{\label{eq:BP5hypothesis}\int_{B_1(0)} |v(x,s)|^2\,dx + \int_{-1}^s\int_{B_1(0)} |\nb v|^2 \, dx \, ds' \le E(s-t_0)^\eta_+, \\
\int_{-1}^s \int_{B_1(0)} |q|^{\frac 32} \, dx\,ds' \le E(s-t_0)^{\frac 34 \eta}_+,}
where $(\cdot)_+ := \max(\cdot,0)$, the conditions
\[\sup_{s\in(-1,0)} |s-t_0|^{\frac 12} \|a(\cdot,s)\|_{L^\I(B_1(0))} \le \varepsilon_*\]
and 
\[\int_{Q_1(0,0)} |v|^3 + |q|^{\frac 32} \,dx\,ds \le \varepsilon_*\]
imply that for all $(\bar x,t)$ in $\bar Q_{1/2}(0,0),$ for all $r\in (0,\frac 14]$,
\[\frac 1 {r^3} \int_{Q_r(\bar x,t)} |v|^3 \,dx\,ds \le C_* \varepsilon_*^{\frac 23}r^{-\de}.\]
\end{theorem}

The last result is a local regularity criteria for the perturbed Navier-Stokes equations. The form we include is a  revision of \cite[Theorem 5]{BaPr2020} which is justified in \cite[Appendix C]{BaPr2020}.

\begin{theorem}\label{thrm:BaPr2020no6}
Let $a$ be a divergence free vector field for which  $\| a \|_{\Kp} < \I$ for some  fixed $p\in (3,\I)$. Let $u\in C_w([0,1];L^2(B_2(0)))$, $\nb u\in L^2(B_2(0) \times (0,1))$ and $p\in L^1((0,1);L^1(B_2(0)))$ satisfy \eqref{eq:pns} distributionally and $u(\cdot,0)=0\in B_2(0)$. Furthermore, assume there exists $\de\in (0,3/2)$ and $S^*\in (0,1/4)$ so that
\[
\sup_{0<r\leq \frac 1 4; (\bar x,t)\in \bar Q_{\frac 1 2}(0,S_*)} r^{\de-5} \int_{Q_r(\bar x,t)} |u|^3\,dx\,ds <\I.
\]
Then, there is a universal constant $\e_{**}\in (0,\I)$ so that if 
\EQ{\label{cond.KatoP}
 \| a(\cdot,s)\|_{\Kp} (s)\leq \e_{**},
}
then $u\in C^{0,\si}_{\text{par}} (B_{1/3}(0)\times (0,S^*))$ for some $\si\in (0,1/2)$.
\end{theorem}
 
In \cite{BaPr2020}, it is shown that a perturbed solution satisfies Proposition \ref{prop:BaPr2020no4}. Then, Theorem \ref{prop:BaPr2020no5} provides H\"older regularity.  

\begin{proof}[Proof of Theorem \ref{thrm:Bpdecay}]
We decompose our data using \cref{lemma:BpDataDecomp} to obtain $u_0=a_0+b_0$ with
\[ \|a_0\|_{\Bp(\R^3)}<\varepsilon \qquad |b_0(x)|\le \frac{K_0(\varepsilon, u_0,\la)}{|x|},\]
Next, by taking $\varepsilon$ small enough to apply \cref{prop:BaPr2020no4}, we obtain a solution $a$ to \eqref{eq:NS} which remains small in $\Bp$, $\Kp$, and  $\cK_\I$, and is unique in the class of solutions with sufficiently small $\Kp$ norm.

Examining $b_0$, we can choose $R_\varepsilon>0$ such that $\|b_0\|_{L^\I(\R^3\setminus B_{R_\varepsilon}(0))}<\varepsilon$. Choose the smallest $k$ such that $A_k=\{x: \la^k\le|x|<\la^{k+1}\}\subset \R^3\setminus B_{R_\varepsilon}(0)$ and for every $x\in A_k$, 
$B_2(x)\subset \R^3\setminus B_{R_\varepsilon}(0)$. Next, fix $x_0\in A_k$  and we localize $b_0$ to $B_2(x_0)$ as follows:
Let $\phi\in C^\I_c(\R^3)$ satisfy $\phi \equiv 1$ on $B_1(x_0)$, $\supp \phi \subset B_2(x_0)$, and $\|\phi\|_{L^\I}+\|\nb \phi \|_{L^\I} \lesssim 1$. 

We will eventually apply a local smoothing result to a solution to the perturbed NS starting from $b_0$, where we emphasize that $b_0\in L^2_\uloc$ but that this may not be the case for $u_0$. Indeed the basic idea of our proof is that local smoothing for such a solution  will imply that the $|x|^{-1}$ asymptotics of $b_0$ extend to the solution.   
Because there is a background perturbation which is only in $L^\I_t \Bp   \cap \mathcal K_p$, we need to use the treatment of the problem for locally small $\Bp$  data in \cite[Appendix B]{BaPr2020} as opposed to what appears in the main body of \cite{BaPr2020}. 

To make this idea rigorous we first split the data.
Using the Bogovskii map, we know that, for any $1<q<\I$, there is a $\td b_{0,1}\in W_0^{1,q}(B_2(x_0)\setminus B_1(x_0))$ that satisfies
\[\div \td b_{0,1} = b_0 \cdot \nb \phi,\]
and 
\[\|\td b_{0,1}\|_{W^{1,q}}\le C(q)\|b_0 \cdot \nb \phi\|_{L^q}\le C(q)\|b_0
\|_{L^\I}(A_k)< C(q)\varepsilon,\]
where $C(q)$ also depends on $\phi$. Now write $b_0 = b_{0,1}+b_{0,2}$ where $b_{0,1} = \phi b_0-\td b_{0,1}$ and $b_{0,2}=b_0-b_{0,1}$ are divergence free. Additionally, $\supp b_{0,1}\subset B_2(x_0)$ and by \cite[p. 1517]{BaPr2020}, $\|b_{0,2}\|_{L^2_\uloc}\lesssim \| b_0\|_{L^2_\uloc}$. 

We next show $\|b_{0,1}\|_{L^\I}$ is controlled by $\varepsilon$.
By Sobolev embedding,
\[\|\td b_{0,1}\|_{C^{0,\frac 12}(\R^3)} \le C\|\td b_{0,1}\|_{W^{1,6}(\R^3)}\le C\|b_0\|_{L^\I(A_k)} \lesssim \varepsilon.\]
Since $\td b_{0,1} = 0$ on the boundary of $B_2(x_0)$, membership in the H\"older class implies 
\EQ{\|\td b_{0,1}\|_{L^\I (B_2(x_0))} &=\|\td b_{0,1}(x)-\td b_{0,1}(y)\|_{L^\I_x (B_2(x_0))}\\&\lesssim \|\td b_{0,1}\|_{C^{0,\frac 12}(B_2(x_0))} \sup_{x\in B_2(x_0)}|x-y|^{\frac 12} \lesssim \e,}
for $y\in \pd B_2(x_0)$, and, therefore, $\| b_{0,1}\|_{L^\I}\lesssim \e$. 

Next, because $\|b_{0,1}\|_{L^p(\R^3)}$ and $\|a\|_{K_\I}$ can be made arbitrarily small through our selection of $\varepsilon$ and $R_\varepsilon$, we can apply \cref{lemma:Linfty} to solve \eqref{eq:pns} obtaining a solution $b_1$ that satisfies
\[
\partial_t b_1 -\Delta b_1 +b_1\cdot \nb b_1+a\cdot\nb b_1+b_1\cdot\nabla a +\nb \pi_1 = 0 ; \qquad \nb \cdot b_1=0, 
\]
\[\sup_{0<t<1}\|b_1\|_{L^\I}(t) \le 2\|b_{0,1}\|_{L^\I},\]
and
\[\sup_{0<t<1}\|b_1\|_{L^p}(t) \le 2\|b_{0,1}\|_{L^p}.\]
Therefore  $\|b_1\chi_{[0,1]}\|_{\mathcal K_p}+\|b_1\chi_{[0,1]}\|_{\mathcal K_\I}\lesssim \|b_{0,1}\|_{L^p}$.  


By taking $b_{0,1}$ sufficiently small,  the preceding observations allow us to conclude that
\EQ{\label{condOnv}
\|(a+b_1)\chi_{[0,1]}\|_{\mathcal K_p}+\|(a+b_1)\chi_{[0,1]}\|_{\mathcal K_\I}<\e.}
Finally, observe that $v=a+b_1$ solves \eqref{eq:NS}.

We now consider 
$b_2$ which solves 
\[\pd_t b_2 - \De b_2 +b_2\cdot \nb b_2 +v\cdot \nb b_2 +b_2\cdot \nb v+\nb \pi_2 =0,\]
where $v=a+b_1$ is a solution to \eqref{eq:NS}.
We are now in \textit{exactly} the same position as in \cite[Proof of Theorem 1]{BaPr2020}, albeit with the $\Bp$ modifications in \cite[Appendix C]{BaPr2020} and the fact that $v$ is not a local energy solution to \eqref{eq:NS}. 
To make sure our reasoning aligns with \cite{BaPr2020}, we need to check that $b_2$ is a local energy solution, regardless of what is true of $v$.   The majority of the remainder of this proof is dedicated to this.

That $(b_2,\pi_2)$ solves the perturbed Navier-Stokes equations around $v=a+b_1$ as distributions follows from the fact the fact that $u$, $a$ and $b_1$ solve their respective problems distributionally. The pressure expansion is similarly justified.

Regarding the second item in the definition of local energy solutions, namely membership in $L^{\I}(0,T; L^2_\loc)\cap L^2 (0,T; \dot H^1_\loc)$, we note that $b_2 = u - (a+b_1)$ and $u$ is a $\Bp$-weak solution. We may also take $a$ to be a $\Bp$-weak solution by weak-strong uniqueness.  
Hence, 
\[
b_2 = b_1 + (u-P_k(u_0)) + (P_k(u_0)-P_k(a_0)) + (P_k(a_0)-a).
\]
The energy estimates associated with $\Bp$-weak solutions imply the $u-P_k(u_0)$ and $P_k(a_0)-a$ satisfy the needed inclusions in the first item of the definition of local energy solutions. Parabolic potential estimates imply $t^{1/2} \nb b_1 \in L^\I_{x,t}$ and so $b_1$ also satisfies these inclusions. For $P_k(u_0)-P_k(a_0)$ we consider the expansion
\[
P_k(u_0)-P_{k-1}(u_0) + P_{k-1}(u_0)-\cdots -e^{t\Delta}u_0 + e^{t\Delta}a_0 - e^{t\Delta}a_0+\cdots +P_{k-1}(a_0)-P_k(a_0). 
\]
Each difference of adjacent Picard iterates enjoys the same energy bound as the difference $u-P_k(u_0)$, which is stated in \cite[(2.37)]{AB}. In particular these differences satisfy the desired inclusion.  This leaves 
\[
e^{t\Delta}(u_0-a_0).
\]
Note that $a_0 - u_0 \in L^2_\uloc$. By properties of the heat equation in $L^2_\uloc$ (see, e.g.,~\cite{MaTe}), we  conclude that $e^{t\Delta}(u_0-a_0)$ satisfies the correct inclusion. Putting all this together we have that 
\[
b_2\in L^{\I}(0,T; L^2_\loc)\cap L^2 (0,T; \dot H^1_\loc),
\]
for $T$ the  existence time of $b_1$. This confirms the second item in the definition of local energy solutions holds. 

The fourth item of the definition of local energy solutions follows from essentially the same remarks. In particular, for a compact set $K$, each difference of Picard iterates goes to zero in $L^2(K)$, and the same is true for $u-P_k(u_0)$ and $P_k(a_0)-a$. Properties of the heat equation imply $e^{t\Delta}(u_0-a_0)\to u_0-a_0$ in $L^2(K)$. Finally, since $b_{0,1}\in L^4$, $b_1 \to b_{0,1}$ in $L^4(K)$ by Lemma \ref{lemma:Linfty}. Taken together, we see that $b_2\to b_{0,2}$ in $L^2(K)$ for every compact set $K$. This also implies the sixth item of  the definition of local energy solutions  holds at $t=0$. For $t>0$, the sixth item follows from the second item of Definition \ref{defn:BesovSoln} along with continuity properties of solutions to the heat equation, as well as the continuity  properties of $a$ in Lemma \ref{lemma:Linfty}.

That the pressure belongs to $L^{3/2}_\loc (\R^4_+)$ is inherited from the fact that this holds for the pressures of $u$ and $a$ as they are $\Bp$-weak solutions, as well as the fact it holds for the pressure of $b_1$---indeed $b_{0,1}\in L^2_\uloc$ so it satisfies all properties of (perturbed) local energy solutions by weak strong uniqueness. 

 Regarding a local energy inequality for $b_2$, observe that $b_2$ solves the perturbed Navier-Stokes equations around $v= a+b_1$ where $v$ solves the Navier-Stokes equations. The local energy inequality for $b_2$ is formally obtained by testing the $b_2$-equation against $b_2\phi$ where $\phi$ is a non-negative test function and integrating by parts. As $b_2$ is not known to be smooth, this calculation cannot be carried out directly but must be reduced to statements about $u$ and $v$. This is sensitive because we do not have energy estimates down to $t=0$ for $u$. Fortunately, 
 the definition of the local energy inequality for perturbed NS in \cite{BaPr2020}, see \cite[Display (25)]{BaPr2020}, uses test functions which are supported in $C_c^\I(Q_1)$ where $Q_1$ is an open space-time cylinder. In our case, the functions in $C_c^\I$   are supported away from $t=0$. This means that all assertions about $u$ and $v$ with respect to the local energy inequality are the same  in our context compared to \cite{BaPr2020}---the fact that $u_0\notin L^2_\loc$ plays no role here. In particular, we can conclude that $b_2$ satisfies the needed local energy inequality in the same way as in \cite{BaPr2020}. As the calculation is rather lengthy and is omitted from \cite{BaPr2020}, we include it for convenience.

Let $\phi \in C^\I_c(\R^4_+)$. Observe that if $v= a+b_1$, then $v$ solves the Navier-Stokes equations distributionally and, in view of the regularity of $a$ and $b_1$ is smooth for $t>0$.

Observe that $u$ satisfies the local energy inequality by assumption while $v$ satisfies the local energy equality due to its regularity. Noting that $b_2=u-v$, we obtain
\EQ{\label{lei1}
2\iint |\nb b_2|^2\phi \,dx\,dt &= 2\iint ( |\nb u|^2+|\nb v|^2 - 2 \nb u:\nb v  )\phi\,dx\,dt
\\&\leq \iint |u|^2 (\partial_t \phi +\Delta \phi) \,dx\,dt+ \iint (|u|^2+2p_u) (u\cdot\nb \phi)\,dx\,dt
\\&+\iint |v|^2 (\partial_t \phi +\Delta \phi)\,dx\,dt + \iint (|v|^2+2p_v) (v\cdot\nb \phi)\,dx\,dt
\\&-4 \iint \nb u:\nb v  \,dx\,dt.
}
We have 
\[
\iint |u|^2 ((\partial_t \phi +\Delta \phi) +|v|^2 (\partial_t \phi +\Delta \phi))\,dx\,dt = \iint ( |b_2|^2 (\partial_t \phi +\Delta \phi) + 2 u\cdot v (\partial_t \phi +\Delta \phi))\,dx\,dt
\]
Hence,
\EQ{\label{lei2}
&2\iint |\nb b_2|^2\phi \,dx\,dt 
\\&\leq \iint |b_2|^2 (\partial_t \phi +\Delta \phi) \,dx\,dt + 2\iint u\cdot v (\partial_t \phi +\Delta \phi)\,dx\,dt-4 \iint \nb u:\nb v \phi \,dx\,dt
\\&+ \iint (|u|^2+2p_u) (u\cdot\nb \phi)\,dx\,dt + \iint (|v|^2+2p_v) (v\cdot\nb \phi)\,dx\,dt.
}
We now consider the equations one gets from testing $\partial_t u $ against $2v\phi$ and $\partial_t v$ against $2u\phi$ and summing. In the first case we get after integrating by parts, which is justified because $2v\phi$ is a test function,  
\EQ{
0=&\iint (\partial_t u-\Delta u +u\cdot \nb u +\nb p_u)(2\phi v)\,dx\,dt 
\\&= -\iint ( 2u\partial_t v \phi + 2 u\cdot v \partial_t \phi) \,dx\,dt
\\&+ \iint 2 \nb u : \nb v \phi \,dx\,dt+2 \iint \partial_i u_j \partial_i \phi v_j \,dx\,dt
\\&+ \iint (u\cdot \nb u +\nb p_u)(2\phi v)\,dx\,dt,
}
where we are summing over the indices $i$ and $j$ where they appear.
Adding the equation we get from testing $\partial_t v$ against $2u\phi$ to this we obtain
\EQ{
0&=-\iint 2 u\cdot v \partial_t \phi \,dx\,dt 
\\&+ \iint 4 \nb u : \nb v \phi\,dx\,dt +2 \iint \partial_i u_j \partial_i \phi v_j \,dx\,dt +2 \iint (\nb v \cdot  \nb \phi) u\,dx\,dt
\\&+ \iint (u\cdot \nb u +\nb p_u)(2\phi v)\,dx\,dt +  \iint (v\cdot \nb v +\nb p_v)(2\phi u)\,dx\,dt.
}
Observe that  
\[
2 \iint \partial_i u_j \partial_i \phi v_j \,dx\,dt +2 \iint\partial_i v_j \partial_i \phi u_j \,dx\,dt= -2 \int u\cdot v \Delta \phi \,dx\,dt.
\]
Adding the above to \eqref{lei2} gives
\EQ{\label{lei3}
&2\iint |\nb b_2|^2\phi \,dx\,dt 
\\&\leq \iint |b_2|^2 (\partial_t \phi +\Delta \phi) \,dx\,dt  
\\&+ \iint (u\cdot \nb u +\nb p_u)(2\phi v)\,dx\,dt +  \iint (v\cdot \nb v +\nb p_v)(2\phi u)\,dx\,dt
\\&+ \iint (|u|^2+2p_u) (u\cdot\nb \phi)\,dx\,dt + \iint (|v|^2+2p_v) (v\cdot\nb \phi)\,dx\,dt.
}
A direct calculation and adding and subtracting $\iint b_2\cdot \nb v b_2\phi\,dx\,dt$  reveals that
\EQ{
\iint |u|^2 u\cdot \nb \phi\,dx\,ds &= \iint \big( |b_2|^2 b_2\cdot\nb \phi + |b_2|^2 v\cdot\nb \phi - b_2\cdot \nb v b_2\phi\big)\,dx\,dt
\\&+\iint \big( b_2\cdot \nb v b_2\phi\,dx\,dt -|v|^2 u\cdot\nb \phi + 2( u\cdot v )(u\cdot \nb \phi ) \big)\,dx\,ds.
}
The leading three terms on the right-hand side comprise the non-linear part of the local energy inequality for $b_2$, omitting the pressure.
Comparing the above to the non-pressure and non-linear terms on the right-hand side of \eqref{lei3},  
we need to have
\[
\iint\big( b_2\cdot \nb v b_2\phi -|v|^2 u\cdot\nb \phi + 2 (u\cdot v) u\cdot \nb \phi +\underbrace{v\cdot\nb v 2 \phi u  + u\cdot\nb u 2\phi v + |v|^2 v\cdot\nb \phi }_{\text{ from \eqref{lei3}}}\big)\,dx\,ds= 0.
\]
This is clearly the case once we expand $b_2\cdot \nb v b_2\phi$ in terms of $u$ and $v$. A similar calculation applies to the pressure.

At this point, our proof aligns \textit{exactly} with the work in \cite{BaPr2020}. In particular, 
the quantitative estimates \eqref{eq:BP5hypothesis} are deduced in the same way and \eqref{eq:BP5hypothesis0} are satisfied due to \eqref{condOnv}.
 Therefore, there exists a short time, $S_*$  so that 
 \[\int_0^{S_*}\int_{B_1(x_0)} |b_2|^3\,dx\,ds + \int_0^{S_*}\int_{B_1(x_0)} |\pi_2|^{\frac 32} \,dx\,ds\le\varepsilon_*.\]
We then extend $b_2$ by zero to a local suitable energy solution on $B_1(x_0)\times(-1+S_*,S_*)$ and we use \cref{prop:BaPr2020no5} to conclude that
 \[\frac 1{r^3} \int_{Q_r(\bar x,t)}|b_2|^3\,dx\,ds \le C_* \varepsilon^{2/3}\]
 for $(x,t)\in\bar Q_{\frac 12}(x_0,S_*)$, and $0<r<\frac 14$. From this and using \eqref{condOnv}, we can apply Theorem \ref{thrm:BaPr2020no6} to conclude that $b_2\in C_\text{par}^{0,\si}(B_{\frac 13}(x_0)\times[0,S_*))$.

Note that $a$ is defined globally while $b_1$ and $b_2$ have the properties we want in 
$B_{\frac 13}(x_0)\times[0,S_*)$. We can obtain global fields, which we abusively label $b_1$ and $b_2$, from the solutions $b_1$ and $b_2$ constructed above by following the procedure immediately below \eqref{conditionLIST}. This proves the theorem.

\end{proof}

\section*{Acknowledgements}

The research of Z.~Bradshaw was supported in part by the NSF grant DMS-2307097 and the Simons Foundation via a TSM grant (formerly called a collaboration grant).

\end{document}